\documentclass[a4paper,oneside,10pt]{article}%
\usepackage{amsmath}
\usepackage{amsfonts}
\usepackage{amssymb}
\usepackage{graphicx}
\usepackage{color}
\usepackage{bbm}
\usepackage[square,numbers,sort&compress]{natbib}%
\providecommand{\U}[1]{\protect \rule{.1in}{.1in}}
\pdfoutput=1

\newtheorem{theorem}{Theorem}[section]

\newtheorem{definition}[theorem]{Definition}

\newtheorem{example}[theorem]{Example}

\newtheorem{lemma}[theorem]{Lemma}

\newtheorem{proposition}[theorem]{Proposition}
\newtheorem{remark}[theorem]{Remark}

\newenvironment{proof}[1][Proof]{\noindent \textbf{#1.} }{\  \rule{0.5em}{0.5em}}
\numberwithin{equation}{section}
\begin{document}

\title{On the rate of convergence for an $\alpha$-stable central limit theorem under
sublinear expectation}
\author{Mingshang Hu\thanks{Zhongtai Securities Institute for Financial Studies,
Shandong University, Jinan, Shandong 250100, PR China. humingshang@sdu.edu.cn.
Research supported by National Natural Science Foundation of China (No.
12326603, 11671231) and National Key R\&D Program of China (No.
2018YFA0703900).}
\and Lianzi Jiang\thanks{College of Mathematics and Systems Science, Shandong
University of Science and Technology, Qingdao, Shandong 266590, PR China.
jianglianzi95@163.com. Research supported by National Natural Science
Foundation of Shandong Province (No. ZR2023QA090)}
\and Gechun Liang\thanks{Department of Statistics, The University of Warwick,
Coventry CV4 7AL, U.K. g.liang@warwick.ac.uk. Research supported by National
Natural Science Foundation of China (No. 12171169) and Laboratory of
Mathematics for Nonlinear Science, Fudan University.}}
\date{}
\maketitle

\textbf{Abstract}.
In this paper, we propose a monotone approximation scheme
for a class of fully nonlinear degenerate partial integro-differential
equations (PIDEs) which characterize the nonlinear $\alpha$-stable L\'{e}vy
processes under sublinear expectation space with $\alpha \in(1,2)$. We further
establish the error bounds for the monotone approximation scheme. This in turn
yields an explicit Berry-Esseen bound and convergence rate for the $\alpha$-stable
central limit theorem under sublinear expectation.

%


\textbf{MSC-classification}. 60F05, 60E07, 65M15.

\section{Introduction}

Motivated by measuring risks under model uncertainty, Peng
\cite{P2004,P2007,P20081,P2010} introduced the notion of sublinear expectation
space, called $G$-expectation space. The $G$-expectation theory has been
widely used to evaluate random outcomes, not using a single probability
measure, but using the supremum over a family of possibly mutually singular
probability measures. One of the fundamental results in this theory is the
celebrated Peng's robust central limit theorem introduced in
\cite{P20082,P2010}.
The corresponding convergence rate was an open problem until recently. The
first convergence rate was established by Song \cite{FPSS2019,Song2020} using
Stein's method and later by Krylov \cite{Krylov2020} using stochastic control
method under different model assumptions. More recently, Huang and Liang
\cite{HL2020} studied the convergence rate of a more general central limit
theorem via a monotone approximation scheme for the $G$-equation.

On the other hand, the nonlinear L\'{e}vy processes have been studied by Hu
and Peng \cite{HP2021} and Neufeld and Nutz \cite{NN2017}. For $\alpha
\in(1,2)$, they consider a nonlinear $\alpha$-stable L\'{e}vy process
$(X_{t})_{t\geq0}$ defined on a sublinear expectation space $(\Omega
,\mathcal{H},\mathbb{\hat{E})}$, whose local characteristics are described by
a set of L\'{e}vy triplets $\Theta=\{(0,0,F_{k_{\pm}}):k_{\pm}\in K_{\pm}\}$,
where $K_{\pm}\subset(\lambda_{1},\lambda_{2})$ for some $\lambda_{1}%
,\lambda_{2}\geq0$ and $F_{k_{\pm}}(dz)$ is the $\alpha$-stable L\'{e}vy
measure
\[
F_{k_{\pm}}(dz)=\frac{k_{-}}{|z|^{\alpha+1}}\mathbf{1}_{(-\infty
,0)}(z)dz+\frac{k_{+}}{|z|^{\alpha+1}}\mathbf{1}_{(0,+\infty)}(z)dz\text{.}%
\]
Such a nonlinear $\alpha$-stable L\'{e}vy process can be characterized via a
fully nonlinear partial integro-differential equation (PIDE). For any $\phi \in
C_{b,Lip}(\mathbb{R})$, Neufeld and Nutz \cite{NN2017} proved the following
representation result
\[
u(t,x):=\mathbb{\hat{E}}[\phi(x+X_{t})]\text{, \ }(t,x)\in \lbrack
0,T]\times \mathbb{R},
\]
where $u$ is the unique viscosity solution of the fully nonlinear PIDE%

\begin{equation}
\left \{
\begin{array}
[c]{ll}%
\displaystyle \partial_{t}u(t,x)-\sup \limits_{k_{\pm}\in K_{\pm}}\left \{
\int_{\mathbb{R}}\delta_{z}u(t,x)F_{k_{\pm}}(dz)\right \}  =0, & (t,x)\in
(0,T]\times \mathbb{R},\\
\displaystyle u(0,x)=\phi(x),\text{ \ }x\in \mathbb{R}, &
\end{array}
\right.  \label{1.1}%
\end{equation}
with $\delta_{z}u(t,x):=u(t,x+z)-u(t,x)-D_{x}u(t,x)z$. In contrast to the
fully nonlinear PIDEs studied in the PDE literature, (\ref{1.1}) is driven by
a family of $\alpha$-stable L\'evy measures rather than a single L\'evy
measure. Moreover, since $F_{k_{\pm}}(dz)$ possesses a singularity at the
origin, the integral term degenerates and \eqref{1.1} is a degenerate equation.

The corresponding generalized central limit theorem for $\alpha$-stable random
variables under sublinear expectation was established by Bayraktar and Munk
\cite{BM2016}. For this, let $( \xi_{i})_{i=1}^{\infty}$ be a sequence of
i.i.d. $\mathbb{R}$-valued random variables on a sublinear expectation space
$(\Omega,\mathcal{H},\mathbb{\tilde{E})}$. After proper normalization,
Bayraktar and Munk proved that
\[
\lim_{n\rightarrow \infty}\mathbb{\tilde{E}}\left[  \phi \left(  \sum_{i=1}%
^{n}\frac{\xi_{i}}{\sqrt[\alpha]{n}}\right)  \right]  =\mathbb{\mathbb{\hat
{E}}}[\phi(X_{1})]\text{,}%
\]
for any $\phi \in C_{b,Lip}(\mathbb{R})$. We refer to the above convergence
result as the $\alpha$-stable central limit theorem under sublinear expectation.

Noting that $\mathbb{\hat{E}}[\phi(X_{1})\mathbb{]}=u(1,0)$, where $u$ is the
viscosity solution of (\ref{1.1}), in this work, we study the rate of
convergence for the $\alpha$-stable central limit theorem under sublinear
expectation via the numerical analysis method for the nonlinear PIDE (\ref{1.1}). To
do this, we first construct a sublinear expectation space $(\mathbb{R},C_{Lip}(\mathbb{R}),\mathbb{\tilde{E}})$ and introduce a random variable
$\xi$ on this space. For given $T>0$ and $\Delta \in(0,1)$, using the random variable
$\xi$ under $\mathbb{\tilde{E}}$ as input, we define a discrete scheme
$u_{\Delta}:[0,T]\times \mathbb{R\rightarrow R}$ to approximate $u$ by
\begin{equation}\begin{array}
[c]{l}u_{\Delta}(t,x)=\phi(x),\text{ \ if }t\in \lbrack0,\Delta),\\
u_{\Delta}(t,x)=\mathbb{\tilde{E}}[u_{\Delta}(t-\Delta,x+\Delta^{1/\alpha}\xi)],\text{ \ if }t\in \lbrack \Delta,T].
\end{array}
\label{1.0}\end{equation}
Taking $T=1$ and $\Delta=\frac{1}{n}$, we can recursively apply the above
scheme to obtain\[
\mathbb{\tilde{E}}\left[  \phi \left(  \sum_{i=1}^{n}\frac{\xi_{i}}{\sqrt[\alpha]{n}}\right)  \right]=u_{\Delta}(1,0).
\]
In this way, the convergence rate of the $\alpha$-stable central limit theorem is
transformed into the convergence rate of the discrete scheme (\ref{1.0}) for approximating the nonlinear PIDE (\ref{1.1}).

The basic framework for convergence of numerical schemes to viscosity
solutions of HJB equations was established by Barles and Souganidis
\cite{BS1991}. They showed that any monotone, stable and consistent
approximation scheme converges to the correct solution, provided that there
exists a comparison principle for the limiting equation. The corresponding
convergence rate was first obtained\ by Krylov, who introduced the shaking
coefficients technique to construct a sequence of smooth
subsolutions/supersolutions in \cite{Krylov1997,Krylov1999,Krylov2000}. This
technique was further developed by Barles and Jakobsen to general monotone
approximation schemes (see \cite{BJ2002,BJ2005,BJ2007} and references therein).

The design and analysis of numerical schemes for nonlinear PIDEs is a
relatively new area of research. For nonlinear degenerate PIDEs driven by a
family of $\alpha$-stable L\'{e}vy measures, there are no general results
giving error bounds for numerical schemes. Most of existing results in the PDE
literature only deal with a single L\'{e}vy measure and its finite difference
method, e.g., \cite{BCJ2019,BJK2008,BJK2010,JKC2008}. One exception is
\cite{CRR2016} which considers a nonlinear PIDE driven by a set of tempered
$\alpha$-stable L\'{e}vy measures for $\alpha \in(0,1)$ by using the finite
difference method.

To derive the error bounds for the discrete scheme (\ref{1.0}), the key step
is to interchange the roles of the discrete scheme and the original equation
when the approximate solution has enough regularity. The classical regularity
estimates of the approximate solution depend on the finite variance of random
variables. Since $\xi$ has an infinite variance, the method developed in
\cite{Krylov2020} cannot be applied to $u_{\Delta}$. To overcome this
difficulty, by introducing a truncated discrete scheme $u_{\Delta,N}$ related
to a truncated random variable $\xi^{N}$ with finite variance, we construct a
new type of regularity estimates of $u_{\Delta,N}$, which plays a pivotal role
in establishing the space and time regularity properties for $u_{\Delta}$.
With the help of a precise estimate of the truncation $\mathbb{\tilde{E}}%
[|\xi-\xi^{N}|]$, a novel estimate for $|u_{\Delta}-u_{\Delta,N}|$ is
obtained. By choosing a proper $N$, we then establish the regularity estimates
for $u_{\Delta}$. Together with the concavity of (\ref{1.1}) and (\ref{1.0})
and the regularity estimates of their solutions, we are able to interchange
their roles, and thus derive the error bounds for the discrete scheme. To the
best of our knowledge, this is the first error bounds for the numerical
schemes of fully nonlinear PIDEs associated with a family of $\alpha$-stable
L\'{e}vy measures, which in turn provides a nontrivial convergence rate result
for the $\alpha$-stable central limit theorem under sublinear expectation.

On the other hand, the classical probability literature mainly deals with
$\Theta$ as a singleton, so $(X_{t})_{t\geq0}$ becomes a classical L\'{e}vy
process with triplet $\Theta$, and $X_{1}$ is an $\alpha$-stable random
variable. The corresponding convergence rate of the classical $\alpha$-stable
central limit theorem (with $\Theta$ as a singleton) has been studied in the
Kolmogorov distance (see, e.g.,
\cite{DN2002,H19811,H19812,HL2015,IL1971,JP1998}) and in the Wasserstein-1
distance or the smooth Wasserstein distance (see, e.g.,
\cite{AMPS2017,CGS2011,CX2019,JLL2020,NP2012,X2019}). The first type is proved
by the characteristic functions which do not exist in the sublinear framework,
while the second type relies on Stein's method which also fails under the
sublinear setting.

The rest of the paper is organized as follows. In Section 2, we review some
necessary results about sublinear expectation and $\alpha$-stable L\'{e}vy
processes. In Section 3, we list the assumptions and our main results, the
convergence rate of the $\alpha$-stable random variables under sublinear
expectation. We present two examples to illustrate our results in Section 4.
Finally, by using the monotone scheme method, the proof of our main result is
given in Section 5.

\section{Preliminaries}

In this section, we recall some basic results of sublinear expectation and
$\alpha$-stable L\'{e}vy processes, which are needed in the sequel. For more
details, we refer the reader to \cite{BM2016,NN2017,P2007,P2010} and the
references therein.

We start with some notation. Let $C_{Lip}(\mathbb{R}^{n})$ be the space of
Lipschitz functions on $\mathbb{R}^{n}$, and $C_{b,Lip}(\mathbb{R}^{n})$ be
the space of bounded Lipschitz functions on $\mathbb{R}^{n}$. For any subset
$Q\subset \lbrack0,T]\times \mathbb{R}$ and for any bounded function on $Q$, we
define the norm $|\omega|_{0}:=\sup_{(t,x)\in Q}|\omega(t,x)|.$ We also use
the following spaces: $C_{b}(Q)$ and $C_{b}^{\infty}(Q)$, denoting,
respectively, the space of bounded continuous functions on $Q$ and the space
of bounded continuous functions on $Q$ with bounded derivatives of any order.
For the rest of this paper, we take a nonnegative function $\zeta \in
C^{\infty} (\mathbb{R}^{2})$ with unit integral and support in
$\{(t,x):-1<t<0,|x|<1\}$ and for $\varepsilon \in(0,1)$ let $\zeta
_{\varepsilon}(t,x)=\varepsilon^{-3}\zeta(t/\varepsilon^{2},x/\varepsilon)$.

\subsection{Sublinear expectation}

Let $\mathcal{H}$ be a linear space of real valued functions defined on a set
$\Omega$ such that if $X_{1},\ldots,X_{n}$ $\in \mathcal{H}$, then
$\varphi(X_{1},\ldots,X_{n})\in \mathcal{H}$ for each $\varphi \in
C_{Lip}(\mathbb{R}^{n})$.

\begin{definition}
A functional $\mathbb{\hat{E}}$: $\mathcal{H}\rightarrow \mathbb{R}$ is called
a sublinear expectation: if for all $X,Y \in \mathcal{H}$, it satisfies the
following properties:

\begin{description}
\item[(i)] Monotonicity: If $X\geq Y$ then $\mathbb{\hat{E}}\left[  X\right]
\geq \mathbb{\hat{E}}\left[  Y\right]  ;$

\item[(ii)] Constant preservation: $\mathbb{\hat{E}}\left[  c\right]  =c$ for
any $c\in \mathbb{R};$

\item[(iii)] Sub-additivity: $\mathbb{\hat{E}}\left[  X+Y\right]
\leq \mathbb{\hat{E}}\left[  X\right]  +\mathbb{\hat{E}}\left[  Y\right]  ;$

\item[(iv)] Positive homogeneity: $\mathbb{\hat{E}}\left[  \lambda X\right]
=\lambda \mathbb{\hat{E}}\left[  X\right]  $ for each $\lambda>0.$
\end{description}
\end{definition}

The triplet $(\Omega,\mathcal{H},\mathbb{\hat{E})}$ is called a sublinear
expectation space. From the definition of the sublinear expectation
$\mathbb{\hat{E}}$, the following results can be easily obtained.

\begin{proposition}
\label{Prop^E Plimi}For $X,Y$ $\in \mathcal{H}$, we have

\begin{description}
\item[(i)] If $\mathbb{\hat{E}}\left[  X\right]  =-\mathbb{\hat{E}}\left[
-X\right]  $, then $\mathbb{\hat{E}}\left[  X+Y\right]  =\mathbb{\hat{E}%
}\left[  X\right]  +\mathbb{\hat{E}}\left[  Y\right]  ;$

\item[(ii)] $|\mathbb{\hat{E}}\left[  X\right]  -\mathbb{\hat{E}}\left[
Y\right]  |\leq \mathbb{\hat{E}}\left[  \left \vert X-Y\right \vert \right]  ;$

\item[(iii)] $\mathbb{\hat{E}}\left[  \left \vert XY\right \vert \right]
\leq(\mathbb{\hat{E}}\left[  \left \vert X\right \vert ^{p}\right]  )^{1/p}%
\cdot(\mathbb{\hat{E}}\left[  \left \vert Y\right \vert ^{q}\right]  )^{1/q},$
for $1\leq p,q<\infty$ with $\frac{1}{p}+\frac{1}{q}=1.$
\end{description}
\end{proposition}

\begin{definition}
Let $X_{1}$ and $X_{2}$ be two $n$-dimensional random vectors defined
respectively in sublinear expectation spaces $(\Omega_{1},\mathcal{H}%
_{1},\mathbb{\hat{E}}_{1})$ and $(\Omega_{2},\mathcal{H}_{2},\mathbb{\hat{E}%
}_{2})$. They are called identically distributed, denoted by $X_{1}\overset
{d}{=}X_{2}$, if $\mathbb{\hat{E}}_{1}\left[  \varphi(X_{1})\right]
=\mathbb{\hat{E}}_{2}\left[  \varphi(X_{2})\right]  $, for all $\varphi \in
C_{Lip}(\mathbb{R}^{n})$.
\end{definition}

\begin{definition}
In a sublinear expectation space $(\Omega,\mathcal{H},\mathbb{\hat{E})}$, a
random vector $Y=(Y_{1},\ldots,Y_{n})$ $\in \mathcal{H}^{n}$, is said to be
independent from another random vector $X=(X_{1},\ldots,X_{m})\in
\mathcal{H}^{m}$ under $\mathbb{\hat{E}}[\cdot]$, denoted by $Y\perp X$, if
for every test function $\varphi \in C_{Lip}(\mathbb{R}^{m}\times \mathbb{R}%
^{n})$ we have
\[
\mathbb{\hat{E}}\left[  \varphi(X,Y)\right]  =\mathbb{\hat{E}}\left[
\mathbb{\hat{E}}\left[  \varphi(x,Y)\right]  _{x=X}\right]  .
\]
$\bar{X}=(\bar{X}_{1},\ldots,\bar{X}_{m})\in \mathcal{H}^{m}$ is said to be an
independent copy of $X$ if $\bar{X}\overset{d}{=}X$ and $\bar{X}\perp X$.
\end{definition}

More details concerning general sublinear expectation spaces can be referred
to \cite{P20081,P2010} and references therein.

\subsection{$\alpha$-stable L\'{e}vy process}

\begin{definition}
Let $\alpha \in(0,2]$. A random variable $X$ on a sublinear expectation space
$(\Omega,\mathcal{H},\mathbb{\hat{E})}$ is said to be (strictly) $\alpha
$-stable if for all $a,b\geq0$,
\[
aX+bY\overset{d}{=}(a^{\alpha}+b^{\alpha})^{1/\alpha}X,
\]
where $Y$ is an independent copy of $X$.
\end{definition}

\begin{remark}
For $\alpha=1$, $X$ is the maximal random variables discussed in
\cite{HL2014,P20082,P2010}. When $\alpha=2$, $X$ becomes the $G$-normal random
variables introduced by Peng \cite{P2019,P2010}. In this paper, we shall focus
on the case of $\alpha \in(1,2)$ and consider $X$ for a nonlinear $\alpha
$-stable L\'{e}vy process $(X_{t})_{t\geq0}$ in the framework of \cite{NN2017}.
\end{remark}

Let $\alpha \in(1,2)$, $K_{\pm}$ be a bounded measurable subset of
$\mathbb{R}_{+}$, and $F_{k_{\pm}}$ be the $\alpha$-stable L\'{e}vy measure
\[
F_{k_{\pm}}(dz)=\frac{k_{-}}{|z|^{\alpha+1}}\mathbf{1}_{(-\infty
,0)}(z)dz+\frac{k_{+}}{|z|^{\alpha+1}}\mathbf{1}_{(0,+\infty)}(z)dz,
\]
for all $k_{-},k_{+}\in K_{\pm}$, and denote by $\Theta:=\{(0,0,F_{k_{\pm}%
}):k_{\pm}\in K_{\pm}\}$ the set of L\'{e}vy triplets.
From \cite[Theorem 2.1]{NN2017}, we can define a nonlinear $\alpha$-stable L\'{e}vy process
$(X_{t})_{t\geq0}$ with respect to a sublinear expectation
\[
\mathbb{\hat{E}}[\cdot]=\sup \limits_{P\in \mathfrak{B}_{\Theta}}E^{P}[\cdot],
\]
where $E^{P}$ is the usual expectation under the probability measure $P$,
and\ $\mathfrak{B}_{\Theta}\ $is a set of all semimartingales with $\Theta
$-valued differential characteristics.
This means the following:

\begin{description}
\item[(i)] $(X_{t})_{t\geq0}$ is real-valued c\`{a}dl\`{a}g process and
$X_{0}=0$;

\item[(ii)] $(X_{t})_{t\geq0}$ has stationary increments, that is,
$X_{t}-X_{s}$ and $X_{t-s}$ are identically distributed for all $0\leq s\leq
t;$

\item[(iii)] $(X_{t})_{t\geq0}$ has independent increments, that is,
$X_{t}-X_{s}$ is independent from $(X_{s_{1}},\ldots,X_{s_{n}})$ for each
$n\in \mathbb{N}$ and $0\leq s_{1}\leq \cdots \leq s_{n}\leq s$.
\end{description}

In the following, we present some basic lemmas of the $\alpha$-stable L\'{e}vy
process $(X_{t})_{t\geq0}$. We refer to \cite[Lemmas 2.6-2.9]{BM2016} and \cite[Lemmas 5.1-5.3]{NN2017} for the details of the proof.

\begin{lemma}
We have that
\[
\mathbb{\hat{E}}[|X_{1}|]<\infty.
\]

\end{lemma}

\begin{lemma}
For all $\lambda>0$ and $t\geq0$, $X_{\lambda t}$ and $\lambda^{1/\alpha}%
X_{t}$ are identically distributed.
\end{lemma}

\begin{lemma}
\label{DPP}Suppose that $\phi \in C_{b,Lip}(\mathbb{R})$. Then,$\ $for any
$(t,x)\in \lbrack0,T]\times \mathbb{R}$,
\[
u(t,x)=\mathbb{\hat{E}}[\phi(x+X_{t})],
\]
is the unique viscosity solution of the fully nonlinear PIDE%
\begin{equation}
\left \{
\begin{array}
[c]{ll}%
\displaystyle \partial_{t}u(t,x)-\sup \limits_{k_{\pm}\in K_{\pm}}\left \{
\int_{\mathbb{R}}\delta_{z}u(t,x)F_{k_{\pm}}(dz)\right \}  =0, & (t,x)\in
(0,T]\times \mathbb{R},\\
\displaystyle u(0,x)=\phi(x),\text{ \ }x\in \mathbb{R}, &
\end{array}
\right.  \label{u_PIDE}%
\end{equation}
with $\delta_{z}u(t,x):=u(t,x+z)-u(t,x)-D_{x}u(t,x)z$. Moreover, it holds that
for any $0\leq s\leq t\leq T$,
\[
u(t,x)=\mathbb{\hat{E}}[u(t-s,x+X_{s})].
\]

\end{lemma}

\begin{lemma}
\label{u_regularity}Suppose that $\phi \in C_{b,Lip} (\mathbb{R})$. Then the
function $u$ is uniformly bounded by $|\phi|_{0}$ and jointly continuous. More
precisely, for any $t,s\in \lbrack0,T]$ and $x,y\in \mathbb{R}$,
\[
\left \vert u(t,x)-u(s,y)\right \vert \leq C_{\phi,\mathcal{K}}
(|x-y|+|t-s|^{1/\alpha}),
\]
where $C_{\phi,\mathcal{K}}$ is a constant depending only on Lipschitz
constant of $\phi$ and
\[
\mathcal{K}:=\sup \limits_{k_{\pm}\in K_{\pm}}\left \{  \int_{\mathbb{R}
}|z|\wedge|z|^{2}F_{k_{\pm}}(dz)\right \}  <\infty.
\]

\end{lemma}

\section{Main results}

First, we construct a sublinear expectation space and introduce random
variables on it. For each $k_{\pm}\in K_{\pm}\subset(\lambda_{1},\lambda_{2})$
for some $\lambda_{1},\lambda_{2}\geq0$, let $W_{k_{\pm}}$ be a classical mean
zero random variable with a cumulative distribution function (cdf)
\begin{equation}
F_{W_{k_{\pm}}}(z)=\left \{
\begin{array}
[c]{ll}%
\displaystyle \left[  k_{-}/\alpha+\beta_{1,k_{\pm}}(z)\right]  \frac
{1}{|z|^{\alpha}}, & z<0,\\
\displaystyle1-\left[  k_{+}/\alpha+\beta_{2,k_{\pm}}(z)\right]  \frac
{1}{z^{\alpha}}, & z>0,
\end{array}
\right.  \label{2.1}%
\end{equation}
for some functions $\beta_{1,k_{\pm}}:$ $(-\infty,0]$ $\rightarrow \mathbb{R}$
and $\beta_{2,k_{\pm}}:[0,\infty)\rightarrow \mathbb{R}$ such that%
\[
\lim_{z\rightarrow-\infty}\beta_{1,k_{\pm}}(z)=\lim_{z\rightarrow \infty}%
\beta_{2,k_{\pm}}(z)=0.
\]
Define a sublinear expectation $\mathbb{\tilde{E}}$ on $C_{Lip}(\mathbb{R})$
by%
\begin{equation}
\mathbb{\tilde{E}}[\varphi]=\sup_{k_{\pm}\in K_{\pm}}\int_{\mathbb{R}}%
\varphi(z)dF_{W_{k_{\pm}}}(z),\text{ }\forall \varphi \in C_{Lip}(\mathbb{R}).
\label{E^wan}%
\end{equation}
Clearly, $(\mathbb{R},C_{Lip}(\mathbb{R}),\mathbb{\tilde{E}})$ is a sublinear
expectation space. Let $\xi$ be a random variable on this space given by
\[
\xi(z)=z\text{, \ for all }z\in \mathbb{R}.
\]
Since $W_{k\pm}$\ has mean zero, this yields $\mathbb{\tilde{E}}%
[\xi]=\mathbb{\tilde{E}}[-\xi]=0$.



We need the following assumptions, which are motivated by Example 4.2 in
\cite{BM2016}.

\begin{description}
\item[(A1)]
For each $k_{\pm}\in K_{\pm}$, $\beta_{1,k_{\pm}}$ and $\beta_{2,k_{\pm}}$ are
continuously differentiable functions in \eqref{2.1} satisfying\[
\int_{\mathbb{R}}zdF_{W_{k_{\pm}}}(z)=0.
\]

\item[(A2)] There exists a constant $M>0$ such that for any $k_{\pm}\in
K_{\pm}$, the following quantities are less than $M$:%
\[%
\begin{array}
[c]{lll}%
\displaystyle \left \vert \int_{-\infty}^{-1}\frac{\beta_{1,k_{\pm}}%
(z)}{|z|^{\alpha}}dz\right \vert ,\text{\ } &  & \displaystyle \left \vert
\int_{1}^{\infty}\frac{\beta_{2,k_{\pm}}(z)}{z^{\alpha}}dz\right \vert .
\end{array}
\]

\item[(A3)] There exists a constant $q>0$ such that for any $k_{\pm}\in K_{\pm}$
and $\Delta \in(0,1)$, the following quantities are less than $C\Delta^{q}$:
\[%
\begin{array}
[c]{lll}%
\displaystyle|\beta_{1,k_{\pm}}(-\Delta^{-1/\alpha})|,\text{ \  \ } &
\displaystyle \int_{-\infty}^{-1}\frac{|\beta_{1,k_{\pm}}(\Delta^{-1/\alpha
}z)|}{|z|^{\alpha}}dz,\text{ \ } & \displaystyle \int_{-1}^{0}\frac
{|\beta_{1,k_{\pm}}(\Delta^{-1/\alpha}z)|}{|z|^{\alpha-1}}dz,\\
&  & \\
\displaystyle|\beta_{2,k_{\pm}}(\Delta^{-1/\alpha})|,\text{ \ } &
\displaystyle \int_{1}^{\infty}\frac{|\beta_{2,k_{\pm}}(\Delta^{-1/\alpha}%
z)|}{z^{\alpha}}dz,\text{ \ } & \displaystyle \int_{0}^{1}\frac{|\beta
_{2,k_{\pm}}(\Delta^{-1/\alpha}z)|}{z^{\alpha-1}}dz,
\end{array}
\]
where $C>0$ is a constant.
\end{description}

\begin{remark}
Note that by Assumption (A1) alone, the terms in (A2) are finite and the terms
in (A3) approach zero as $\Delta \rightarrow0$. In other words, the content of
(A2) and (A3) is the uniform bounds and the existence of minimum convergence rates.
\end{remark}

\begin{remark}
By (\ref{2.1}), we can write $\beta_{1,k_{\pm}}$ and $\beta_{2,k_{\pm}}$ as%
\begin{align*}
\beta_{1,k_{\pm}}(z)  &  =F_{W_{k\pm}}(z)|z|^{\alpha}-\frac{k_{-}}{\alpha
},\text{ \ }z\in(-\infty,0],\\
\beta_{2,k_{\pm}}(z)  &  =(1-F_{W_{k\pm}}(z))z^{\alpha}-\frac{k_{+}}{\alpha
},\text{ \ }z\in \lbrack0,\infty).
\end{align*}
Under Assumption (A1), it can be checked that for any $k_{\pm}\in K_{\pm}$ the
following quantities are uniformly bounded (we also assume the uniform bound
is $M$):
\[%
\begin{array}
[c]{lll}%
\displaystyle|\beta_{1,k_{\pm}}(-1)|,\text{\ } &  & \displaystyle \int
_{-1}^{0}\frac{|-\beta_{1,k_{\pm}}^{\prime}(z)z+\alpha \beta_{1,k_{\pm}}%
(z)|}{|z|^{\alpha-1}}dz,\\
&  & \\
\displaystyle|\beta_{2,k_{\pm}}(1)|, &  & \displaystyle \int_{0}^{1}%
\frac{|-\beta_{2,k_{\pm}}^{\prime}(z)z+\alpha \beta_{2,k_{\pm}}(z)|}%
{z^{\alpha-1}}dz.
\end{array}
\]

\end{remark}

\begin{remark}
Under Assumptions (A1)-(A2), it is easy to check that
\[
\mathbb{\tilde{E}}[|\xi|]=\mathbb{\tilde{E}}\left[  \int_{0}^{\infty
}\mathbf{1}_{\{|\xi|>z\}}dz\right]  =\sup_{k_{\pm}\in K_{\pm}}\left \{
\int_{0}^{\infty}P_{k_{\pm}}(|\xi|>z)dz\right \}  ,
\]
where $\{P_{k_{\pm}},$ $k_{\pm}\in K_{\pm}\}$ is the set of probability
measures related to uncertainty distributions $\{F_{W_{k\pm}},k_{\pm}\in
K_{\pm}\}$. Then, it follows that%
\begin{align*}
\mathbb{\tilde{E}}[|\xi|]  &  \leq1+\sup_{k_{\pm}\in K_{\pm}}\left \{  \int
_{1}^{\infty}P_{k_{\pm}}(|\xi|>z)dz\right \} \\
&  \leq1+\sup_{k_{\pm}\in K_{\pm}}\left \{  \frac{k_{-}+k_{+}}{\alpha
(\alpha-1)}+\left \vert \int_{1}^{\infty}\frac{\beta_{2,k_{\pm}}(z)}{z^{\alpha
}}dz\right \vert +\left \vert \int_{1}^{\infty}\frac{\beta_{1,k_{\pm}}%
(-z)}{z^{\alpha}}dz\right \vert \right \}  <\infty.
\end{align*}
Similarly, we know that
\begin{align*}
\mathbb{\tilde{E}}[|\xi|^{2}]  &  \geq \int_{1}^{\infty}P_{k_{\pm}}(|\xi
|>\sqrt{z})dz\\
&  =\int_{1}^{\infty}\frac{k_{+}/\alpha+\beta_{2,k_{\pm}}(\sqrt{z})}%
{z^{\alpha/2}}dz+\int_{1}^{\infty}\frac{k_{-}/\alpha+\beta_{1,k_{\pm}}%
(-\sqrt{z})}{z^{\alpha/2}}dz=\infty.
\end{align*}

\end{remark}

Let $(\xi_{i})_{i=1}^{\infty}$ be a sequence of i.i.d. $\mathbb{R}$-valued
random variables defined on $(\mathbb{R},C_{Lip}(\mathbb{R}),\mathbb{\tilde
{E}})$ in the sense that $\xi_{1}=\xi$, $\xi_{i+1}\overset{d}{=}\xi_{i}$ and
$\xi_{i+1}\perp(\xi_{1},\xi_{2},\ldots,\xi_{i})$ for each $i\in \mathbb{N}$,
and denote
\begin{equation}
\bar{S}_{n}:=\sum_{i=1}^{n}\frac{\xi_{i}}{\sqrt[\alpha]{n}}. \label{Sn}%
\end{equation}

Now we state our first main result. 

\begin{theorem}
\label{main theorem}Suppose that (A1)-(A3) hold. Let $(\bar{S}_{n})_{n=1}^{\infty
}$ be a sequence defined in (\ref{Sn}), $(X_{t})_{t\geq0}$ be a nonlinear
$\alpha$-stable L\'{e}vy process with the characteristic set $\Theta$.
Then, for any $\phi \in C_{b,Lip}(\mathbb{R})$
\begin{equation}
\big|\mathbb{\tilde{E}}[\phi(\bar{S}_{n})]-\mathbb{\hat{E}}[\phi(X_{1})]\big|\leq
C_{0}n^{-\Gamma(\alpha,q)},\label{CLT rate}\end{equation}
where $\Gamma(\alpha,q)=\min \{ \frac{1}{4},\frac{2-\alpha}{2\alpha},\frac{q}{2}\}$ with $q>0$ given in (A3), and $C_{0}$ is a constant
depending on the Lipschitz constant of $\phi$, which will be given in Theorem
\ref{main theorem 2}.
\end{theorem}

\begin{remark}
The classical $\alpha$-stable central limit theorem (see, for example, Ibragimov and Linnik
\cite[Theorem 2.6.7]{IL1971}) states that for a classical mean-zero random variable $\xi
_{1}$, the sequence $\bar{S}_{n}$ converges in law to $X_{1}$ as
$n\rightarrow \infty$, if and only if the cdf of $\xi$ has the form given in
(\ref{2.1}), where $(X_{t})_{t\geq0}$ is a classical L\'{e}vy process with
triplet $(0,0,F_{k_{\pm}})$. In the framework of sublinear expectation,
sufficient conditions for the $\alpha$-stable central limit theorem are given
in Bayraktar and Munk \cite{BM2016}. They show that, for a mean-zero random
variable $\xi_{1}$ under the sublinear expectation $\mathbb{\tilde{E}}$
defined above, $\bar{S}_{n}$ converges in law to $X_{1}$ as $n\rightarrow
\infty$, where $(X_{t})_{t\geq0}$ is a nonlinear L\'{e}vy process with triplet
set $\Theta$. In this paper, Theorem \ref{main theorem} further provides an explicit
convergence rate of the limit theorem in \cite{BM2016}, which can be seen as a special $\alpha$-stable central limit
theorem under the sublinear expectation.
\end{remark}


\begin{remark}
Assumptions $(A1)$-$(A3)$ are sufficient conditions for Theorem 3.1 in
\cite{BM2016}. Indeed, by Proposition 2.10 of \cite{BM2016}, we know that for
any $0<h<1$, $u\in C_{b}^{1,2}([h,1+h]\times \mathbb{R})$. Under Assumptions
$(A1)$-$(A3)$, by using part II in Proposition \ref{prop} (iii) (see Section
5), one gets for any $\phi \in C_{b,Lip}(\mathbb{R})$ and $0<h<1$,
\begin{equation}
n\bigg \vert \mathbb{\tilde{E}}\big[\delta_{n^{-1/\alpha}\xi_{1}}%
v(t,x)\big]-\frac{1}{n}\sup \limits_{k_{\pm}\in K_{\pm}}\bigg \{ \int
_{\mathbb{R}}\delta_{z}v(t,x)F_{k_{\pm}}(dz)\bigg \} \bigg \vert \rightarrow
0\label{B condition}%
\end{equation}
uniformly on $[0,1]\times \mathbb{R}$ as $n\rightarrow \infty$, where $v$ is the
unique viscosity solution of%
\[
\left \{
\begin{array}
[c]{ll}%
\displaystyle \partial_{t}v(t,x)+\sup \limits_{k_{\pm}\in K_{\pm}}\left \{
\int_{\mathbb{R}}\delta_{z}v(t,x)F_{k_{\pm}}(dz)\right \}  =0, & (-h,1+h)\times
\mathbb{R},\\
\displaystyle v(1+h,x)=\phi(x),\text{ \ }x\in \mathbb{R}. &
\end{array}
\right.
\]
In addition, the necessary conditions for the $\alpha$-stable central limit
theorem under sublinear expectation are still unknown.
\end{remark}

\section{Two examples}

In this section, we shall give two examples to illustrate\ our results.

\begin{example}
\label{exp1}Let $( \xi_{i})_{i=1}^{\infty}$ be a sequence of i.i.d.
$\mathbb{R}$-valued random variables defined on $(\mathbb{R},C_{Lip}%
(\mathbb{R}),\mathbb{\tilde{E}})$ with cdf (\ref{2.1}) satisfying
$\beta_{1,k_{\pm}}(z)=0$ for $z\leq-1$ and $\beta_{2,k_{\pm}}(z)=0$ for
$z\geq1$ with $\lambda_{2}<\frac{\alpha}{2}$. The exact expressions of
$\beta_{1,k_{\pm}}(z)$ and $\beta_{2,k_{\pm}}(z)$ for $0<|z|<1$ are not
specified here, but we require $\beta_{1,k_{\pm}}(z)$ and $\beta_{2,k_{\pm}%
}(z)$ to satisfy Assumption (A1). It is clear that Assumption (A2) holds. In
addition, for each\ $k_{\pm}\in K_{\pm}$ and $\Delta \in(0,1)$
\[
\int_{0}^{1}\frac{|\beta_{2,k_{\pm}}(\Delta^{-1/\alpha}z)|}{z^{\alpha-1}%
}dz=\int_{0}^{\Delta^{1/\alpha}}\frac{|\beta_{2,k_{\pm}}(\Delta^{-1/\alpha
}z)|}{z^{\alpha-1}}dz\leq \frac{c}{2-\alpha}\Delta^{\frac{2-\alpha}{\alpha}},
\]
where $c:=\sup \limits_{z\in(0,1)}|\beta_{2,k_{\pm}}(z)|<\infty$, and similarly
for the negative half-line. This indicates that Assumption (A3) holds with
$q=\frac{2-\alpha}{\alpha}$. According to Theorem \ref{main theorem}, we get
the convergence rate
\[
\big|\mathbb{\tilde{E}}[\phi(\bar{S}_{n})]-\mathbb{\hat{E}}[\phi
(X_{1})]\big|\leq C_{0}n^{-\Gamma(\alpha)},
\]
where $\Gamma(\alpha)=\min \{ \frac{1}{4},\frac{2-\alpha}{2\alpha}\}$.
\end{example}

\begin{example}
\label{exp2}Let $( \xi_{i})_{i=1}^{\infty}$ be a sequence of i.i.d.
$\mathbb{R}$-valued random variables defined on $(\mathbb{R},C_{Lip}%
(\mathbb{R}),\mathbb{\tilde{E}})$ with cdf (\ref{2.1}) satisfying
$\beta_{1,k_{\pm}}(z)=a_{1}|z|^{\alpha-\beta}$ for $z\leq-1$, $\beta
_{2,k_{\pm}}(z)=$ $a_{2}z^{\alpha-\beta}\ $for $z\geq1$ with $\beta>\alpha$
and two proper constants $a_{1}, a_{2}$. The exact expressions of $\beta
_{1,k_{\pm}}(z)$ and $\beta_{2,k_{\pm}}(z)$ for $0<|z|<1$ are not specified
here, but we require that $\beta_{1,k_{\pm}}(z)$ and $\beta_{2,k_{\pm}}(z)$
satisfy Assumption (A1). For simplicity, we will only check the integral along
the positive half-line, and similarly for the negative half-line. Observe
that
\[
\int_{1}^{\infty}\frac{\beta_{2,k_{\pm}}(z)}{z^{\alpha}}dz=\frac{a_{2}}%
{\beta-1},
\]
which shows that (A2) holds. Also, it can be verified that for each\ $k_{\pm
}\in K_{\pm}$ and $\Delta \in(0,1)$
\[
|\beta_{2,k_{\pm}}(\Delta^{-1/\alpha})|=a_{2}\Delta^{\frac{\beta-\alpha
}{\alpha}},\text{ \  \  \ }\int_{1}^{\infty}\frac{|\beta_{2,k_{\pm}}%
(\Delta^{-1/\alpha}z)|}{z^{\alpha}}dz=\frac{a_{2}}{\beta-1}\Delta^{\frac
{\beta-\alpha}{\alpha}},
\]
and%
\begin{align*}
\int_{0}^{1}\frac{|\beta_{2,k_{\pm}}(\Delta^{-1/\alpha}z)|}{z^{\alpha-1}}dz
&  =\int_{0}^{\Delta^{1/\alpha}}\frac{|\beta_{2,k_{\pm}}(\Delta^{-1/\alpha
}z)|}{z^{\alpha-1}}dz+\int_{\Delta^{1/\alpha}}^{1}\frac{|\beta_{2,k_{\pm}%
}(\Delta^{-1/\alpha}z)|}{z^{\alpha-1}}dz\\
&  \leq \frac{c}{2-\alpha}\Delta^{\frac{2-\alpha}{\alpha}}+a_{2}\Delta
^{\frac{\beta-\alpha}{\alpha}}\int_{\Delta^{1/\alpha}}^{1}z^{1-\beta}dz,
\end{align*}
where $c=\sup \limits_{z\in(0,1)}|\beta_{2,k_{\pm}}(z)|<\infty$. We further
distinguish three cases based on the value of $\beta$.

\begin{enumerate}
\item[(1)] If $\beta=2$, we have
\[
\int_{0}^{1}\frac{|\beta_{2,k_{\pm}}(\Delta^{-1/\alpha}z)|}{z^{\alpha-1}%
}dz\leq \frac{c}{2-\alpha}\Delta^{\frac{2-\alpha}{\alpha}}+a_{2}\Delta
^{\frac{2-\alpha}{\alpha}}\ln \Delta^{-\frac{1}{\alpha}}\leq C\Delta
^{\frac{2-\alpha}{\alpha}-\varepsilon},
\]
where $C=\frac{c}{2-\alpha}+a_{2}$ and any small $\varepsilon>0$.

\item[(2)] If $\alpha<\beta<2$, we have
\[
\int_{0}^{1}\frac{|\beta_{2,k_{\pm}}(\Delta^{-1/\alpha}z)|}{z^{\alpha-1}%
}dz\leq \frac{c}{2-\alpha}\Delta^{\frac{2-\alpha}{\alpha}}+\frac{a_{2}}%
{2-\beta}(\Delta^{\frac{\beta-\alpha}{\alpha}}-\Delta^{\frac{2-\alpha}{\alpha
}})\leq C\Delta^{\frac{\beta-\alpha}{\alpha}},
\]
where $C=\frac{c}{2-\alpha}+\frac{2a_{2}}{2-\beta}$.

\item[(3)] If $\beta>2$, it follows that
\[
\int_{0}^{1}\frac{|\beta_{2,k_{\pm}}(\Delta^{-1/\alpha}z)|}{z^{\alpha-1}%
}dz\leq \frac{c}{2-\alpha}\Delta^{\frac{2-\alpha}{\alpha}}+\frac{a_{2}}%
{\beta-2}(\Delta^{\frac{2-\alpha}{\alpha}}-\Delta^{\frac{\beta-\alpha}{\alpha
}})\leq C\Delta^{\frac{2-\alpha}{\alpha}},
\]
where $C=\frac{c}{2-\alpha}+\frac{2a_{2}}{\beta-2}$.
\end{enumerate}

Then, Assumption (A3) holds with
\[
q=\left \{
\begin{array}
[c]{ll}%
\frac{2-\alpha}{\alpha}-\varepsilon, & \text{if }\beta=2,\\
\frac{\beta-\alpha}{\alpha}, & \text{if }\alpha<\beta<2,\\
\frac{2-\alpha}{\alpha}, & \text{if }\beta>2,
\end{array}
\right.
\]
for any small $\varepsilon>0$. From Theorem \ref{main theorem}, we can
immediately obtain that
\[
\big|\mathbb{\tilde{E}}[\phi(\bar{S}_{n})]-\mathbb{\hat{E}}[\phi
(X_{1})]\big|\leq C_{0}n^{-\Gamma(\alpha,\beta)},
\]
where
\[
\Gamma(\alpha,\beta)=\left \{
\begin{array}
[c]{ll}%
\min \{ \frac{1}{4},\frac{2-\alpha}{2\alpha}-\frac{\varepsilon}{2}\}, &
\text{if }\beta=2,\\
\min \{ \frac{1}{4},\frac{\beta-\alpha}{2\alpha}\}, & \text{if }\alpha
<\beta<2,\\
\min \{ \frac{1}{4},\frac{2-\alpha}{2\alpha}\}, & \text{if }\beta>2,
\end{array}
\right.
\]
with $\varepsilon>0$.
\end{example}

\section{Proof of Theorem \ref{main theorem}: monotone scheme method}

In this section, we shall introduce the numerical analysis tools of nonlinear
partial differential equations to prove Theorem \ref{main theorem}. Noting
that $\mathbb{\hat{E}}[\phi(X_{1})\mathbb{]}=u(1,0)$, where $u$ is the
viscosity solution of (\ref{u_PIDE}), we propose a discrete scheme to
approximate $u$ by merely using the random variable $\xi$ under
$\mathbb{\tilde{E}}$ as input. For given $T>0$ and $\Delta \in(0,1)$, define
$u_{\Delta}:[0,T]\times \mathbb{R\rightarrow R}$\ recursively by
\begin{equation}%
\begin{array}
[c]{l}%
u_{\Delta}(t,x)=\phi(x),\text{ \ if }t\in \lbrack0,\Delta),\\
u_{\Delta}(t,x)=\mathbb{\tilde{E}}[u_{\Delta}(t-\Delta,x+\Delta^{\frac
{1}{\alpha}}\xi)],\text{ \ if }t\in \lbrack \Delta,T].
\end{array}
\label{2.2}%
\end{equation}
From the above recursive process, we can see for each $x\in \mathbb{R}$ and
$n\in \mathbb{N}$ such that $n\Delta \leq T$, $u_{\Delta}(\cdot,x)$ is a
constant on the interval $[n\Delta,(n+1)\Delta \wedge T)$, that is,
\[
u_{\Delta}(t,x)=u_{\Delta}(n\Delta,x),\text{ \ }\forall t\in \lbrack
n\Delta,(n+1)\Delta \wedge T).
\]

By induction (see Theorem 2.1 in \cite{HL2020}), we can derive that for all
$n\in \mathbb{N}$ such that $n\Delta \leq T$ and $x\in \mathbb{R}$
\[
u_{\Delta}(n\Delta,x)=\mathbb{\tilde{E}}\Big[\phi \Big(x+\Delta^{\frac
{1}{\alpha}}\sum_{i=1}^{n}\xi_{i}\Big)\Big].
\]
In particular, taking $T=1$ and $\Delta=\frac{1}{n}$, we have%
\[
u_{\Delta}(1,0)=\mathbb{\tilde{E}[}\phi(\bar{S}_{n})],
\]
and Theorem \ref{main theorem} follows from the following result.


\begin{theorem}
\label{main theorem 2} Suppose that (A1)-(A3) hold and $\phi \in C_{b,Lip}%
(\mathbb{R})$. Then, for any $\left(  t,x\right)  \in \lbrack0,T]\times
\mathbb{R}$,
\[
\left \vert u(t,x)-u_{\Delta}(t,x)\right \vert \leq C_{0}\Delta^{\Gamma
(\alpha,q)},
\]
where the Berry-Esseen constant $C_{0}=L_{0}\vee U_{0}$ with $L_{0}$ and
$U_{0}$ given explicitly in Lemma \ref{lower bound} and Lemma
\ref{upper bound}, respectively, and
\begin{equation}%
\begin{array}
[c]{r}%
\Gamma(\alpha,q)=\min \{ \frac{1}{4},\frac{2-\alpha}{2\alpha},\frac{q}{2}\}.
\end{array}
\label{tau}%
\end{equation}

\end{theorem}

\subsection{Regularity estimates}

To prove Theorem \ref{main theorem 2}, we first need to establish the space
and time regularity properties of $u_{\Delta}$, which are crucial for proving
the convergence of $u_{\Delta}$ to $u$ and determining its convergence rate.
Before showing our regularity estimates of $u_{\Delta}$, denote
\begin{align*}
I_{1,\Delta}  &  =\sup \limits_{k_{\pm}\in K_{\pm}}\bigg \{ \frac{k_{-}+k_{+}%
}{2-\alpha}+2\int_{0}^{1}\frac{|\beta_{1,k_{\pm}}(-\Delta^{-\frac{1}{\alpha}%
}z)|+|\beta_{2,k_{\pm}}(\Delta^{-\frac{1}{\alpha}}z)|}{z^{\alpha-1}}%
dz+|\beta_{1,k_{\pm}}(-\Delta^{-\frac{1}{\alpha}})|+|\beta_{2,k_{\pm}}%
(\Delta^{-\frac{1}{\alpha}})|\bigg \},\\
I_{2,\Delta}  &  =\sup \limits_{k_{\pm}\in K_{\pm}}\bigg \{ \frac{k_{-}+k_{+}%
}{\alpha-1}+\int_{1}^{\infty}\frac{|\beta_{1,k_{\pm}}(-\Delta^{-\frac
{1}{\alpha}}z)|+|\beta_{2,k_{\pm}}(\Delta^{-\frac{1}{\alpha}}z)|}{z^{\alpha}%
}dz+|\beta_{1,k_{\pm}}(-\Delta^{-\frac{1}{\alpha}})|+|\beta_{2,k_{\pm}}%
(\Delta^{-\frac{1}{\alpha}})|\bigg \}.
\end{align*}

\begin{theorem}
\label{u_num_regularity}Suppose that (A1) and (A3) hold and $\phi \in
C_{b,Lip}(\mathbb{R})$. Then,

\begin{description}
\item[(i)] for any $t\in \lbrack0,T]$ and $x,y\in \mathbb{R}$,
\[
\left \vert u_{\Delta}(t,x)-u_{\Delta}(t,y)\right \vert \leq C_{\phi}|x-y|;
\]

\item[(ii)] for any $t,s\in \lbrack0,T]$ and $x\in \mathbb{R}$,%
\[
\left \vert u_{\Delta}(t,x)-u_{\Delta}(s,x)\right \vert \leq C_{\phi}I_{\Delta
}(|t-s|^{1/2}+\Delta^{1/2}),
\]

\end{description}

where $C_{\phi}$ is the Lipschitz constant of $\phi$ and $I_{\Delta}%
=\sqrt{I_{1,\Delta}}+2I_{2,\Delta}$ with $I_{\Delta}<\infty$.
\end{theorem}

Notice that $\mathbb{\tilde{E}}[\xi^{2}]=\infty$, the classical method
developed in Krylov \cite{Krylov2020} fails. To prove Theorem
\ref{u_num_regularity}, for fixed $N>0$, we define $\xi^{N}:=\xi
\mathbf{1}_{\{|\xi|\leq N\}}$ and introduce the following truncated scheme
$u_{\Delta,N}:[0,T]\times \mathbb{R\rightarrow R}$ recursively by%
\begin{equation}%
\begin{array}
[c]{l}%
u_{\Delta,N}(t,x)=\phi(x),\text{ \ if }t\in \lbrack0,\Delta),\\
u_{\Delta,N}(t,x)=\mathbb{\tilde{E}}[u_{\Delta,N}(t-\Delta,x+\Delta^{\frac
{1}{\alpha}}\xi^{N})],\text{ \ if }t\in \lbrack \Delta,T].
\end{array}
\label{3.1}%
\end{equation}

We get the following estimates.

\begin{lemma}
\label{truncted_moment estimate}For each fixed $N>0$, we have
\[
\mathbb{\tilde{E}}[|\xi^{N}|^{2}]=N^{2-\alpha}I_{1,N},
\]
where%
\[
I_{1,N}:=\sup_{k_{\pm}\in K_{\pm}}\left \{  \frac{k_{-}+k_{+}}{2-\alpha}%
+2\int_{0}^{1}\frac{\beta_{1,k_{\pm}}(-zN)+\beta_{2,k_{\pm}}(zN)}{z^{\alpha
-1}}dz-\beta_{1,k_{\pm}}(-N)-\beta_{2,k_{\pm}}(N)\right \}  .
\]

\end{lemma}

\begin{proof}
Using Fubini's theorem, we obtain
\begin{align*}
&  \mathbb{\tilde{E}}[|\xi|^{2}\mathbf{1}_{\{|\xi|\leq N\}}]=\sup_{k_{\pm}\in
K_{\pm}}\left \{  \int_{\mathbb{R}}\bigg(\int_{0}^{z}2rdr\mathbf{1}_{\{|z|\leq
N\}}\bigg)dF_{W_{k\pm}}(z)\right \} \\
&  =\sup_{k_{\pm}\in K_{\pm}}\left \{  \int_{\mathbb{R}}\bigg(\int_{0}^{\infty
}2r\mathbf{1}_{\{0\leq r<z\}}dr-\int_{-\infty}^{0}2r\mathbf{1}_{\{z\leq
r<0\}}dr\bigg)\mathbf{1}_{\{|z|\leq N\}}dF_{W_{k\pm}}(z)\right \} \\
&  =\sup_{k_{\pm}\in K_{\pm}}\left \{  \int_{0}^{N}2r\bigg(\int_{\mathbb{R}%
}\mathbf{1}_{\{r\leq z\leq N\}}dF_{W_{k\pm}}(z)\bigg)dr-\int_{-N}%
^{0}2r\bigg(\int_{\mathbb{R}}\mathbf{1}_{\{-N\leq z<r\}}dF_{W_{k\pm}%
}(z)\bigg)dr\right \} \\
&  =\sup_{k_{\pm}\in K_{\pm}}\left \{  \int_{0}^{N}2r\left(  F_{W_{k\pm}%
}(N)-F_{W_{k\pm}}(r)\right)  dr-\int_{-N}^{0}2r\left(  F_{W_{k\pm}%
}(r)-F_{W_{k\pm}}(-N)\right)  dr\right \}  .
\end{align*}
By changing variables, it is straightforward to check that%
\begin{align*}
&  \int_{0}^{N}2r\left(  F_{W_{k\pm}}(N)-F_{W_{k\pm}}(r)\right)
dr=N^{2-\alpha}\bigg(\frac{k_{+}}{2-\alpha}+2\int_{0}^{1}\frac{\beta
_{2,k_{\pm}}(zN)}{z^{\alpha-1}}dz-\beta_{2,k_{\pm}}(N)\bigg),\\
&  \int_{-N}^{0}2r\left(  F_{W_{k\pm}}(r)-F_{W_{k\pm}}(-N)\right)
dr=N^{2-\alpha}\bigg(\frac{k_{-}}{2-\alpha}+2\int_{0}^{1}\frac{\beta
_{1,k_{\pm}}(-zN)}{z^{\alpha-1}}dz-\beta_{1,k_{\pm}}(-N)\bigg),
\end{align*}
which immediately implies the result.
\end{proof}

\begin{lemma}
\label{truncation error}For each fixed $N>0$, we have
\[
\mathbb{\tilde{E}}[|\xi-\xi^{N}|]=N^{1-\alpha}I_{2,N},
\]
where%
\[
I_{2,N}:=\sup_{k_{\pm}\in K_{\pm}}\left \{  \frac{k_{-}+k_{+}}{\alpha-1}%
+\beta_{1,k_{\pm}}(-N)+\beta_{2,k_{\pm}}(N)+\int_{1}^{+\infty}\frac
{\beta_{1,k_{\pm}}(-zN)+\beta_{2,k_{\pm}}(zN)}{z^{\alpha}}dz\right \}  .
\]

\end{lemma}

\begin{proof}
Notice that
\begin{equation}
\mathbb{\tilde{E}}[|\xi-\xi^{N}|]=\mathbb{\tilde{E}}\big[|\xi|\mathbf{1}%
_{\{|\xi|>N\}}\big]=\sup_{k_{\pm}\in K_{\pm}}\left \{  \int_{\mathbb{R}%
}|z|\mathbf{1}_{\{|z|>N\}}dF_{W_{k\pm}}(z)\right \}  . \label{3.6}%
\end{equation}
Observe by Fubini's theorem that
\begin{equation}%
\begin{split}
\int_{\mathbb{R}}|z|\mathbf{1}_{\{|z|>N\}}dF_{W_{k\pm}}(z)  &  =\int
_{0}^{\infty}\int_{\mathbb{R}}\mathbf{1}_{\{0\leq r<|z|\}}\mathbf{1}%
_{\{|z|>N\}}dF_{W_{k\pm}}(z)dr\\
&  =\int_{N}^{\infty}\int_{\mathbb{R}}\mathbf{1}_{\{|z|>r\}}dF_{W_{k\pm}%
}(z)dr+N\int_{\mathbb{R}}\mathbf{1}_{\{|z|>N\}}dF_{W_{k\pm}}(z)\\
&  =\int_{N}^{\infty}\left(  1-F_{W_{k\pm}}(r)+F_{W_{k\pm}}(-r)\right)
dr+N\left(  1-F_{W_{k\pm}}(N)+F_{W_{k\pm}}(-N)\right)  .
\end{split}
\label{3.7}%
\end{equation}
Together with (\ref{3.6}) and (\ref{3.7}), we obtain that
\[
\mathbb{\tilde{E}}[|\xi-\xi^{N}|]=\sup_{k_{\pm}\in K_{\pm}}\left \{
\frac{k_{-}+k_{+}}{\alpha-1}N^{1-\alpha}+N^{1-\alpha}\left(  \beta_{1,k_{\pm}%
}(-N)+\beta_{2,k_{\pm}}(N)\right)  +\int_{N}^{\infty}\frac{\beta_{1,k_{\pm}%
}(-r)+\beta_{2,k_{\pm}}(r)}{r^{\alpha}}dr\right \}  .
\]
By changing variables, we immediately conclude the proof.
\end{proof}

\begin{lemma}
\label{u_N regularity}Suppose that $\phi \in C_{b,Lip}(\mathbb{R})$. Then,

\begin{description}
\item[(i)] for any $k\in \mathbb{N}$ such that $k\Delta \leq T$ and
$x,y\in \mathbb{R}$,
\[
\left \vert u_{\Delta,N}(k\Delta,x)-u_{\Delta,N}(k\Delta,y)\right \vert \leq
C_{\phi}|x-y|;
\]

\item[(ii)] for any $k\in \mathbb{N}$ such that $k\Delta \leq T$ and
$x\in \mathbb{R}$,
\[
\left \vert u_{\Delta,N}(k\Delta,x)-u_{\Delta,N}(0,x)\right \vert \leq C_{\phi
}\big(  (I_{1,N})^{\frac{1}{2}}N^{\frac{2-\alpha}{2}}\Delta^{\frac{2-\alpha
}{2\alpha}}+I_{2,N}N^{1-\alpha}\Delta^{\frac{1-\alpha}{\alpha}}\big)
(k\Delta)^{\frac{1}{2}},
\]

\end{description}

where $C_{\phi}$ is the Lipschitz constant of $\phi$, and $I_{1,N}$, $I_{2,N}$
are given in Lemmas \ref{truncted_moment estimate} and \ref{truncation error}, respectively.
\end{lemma}

\begin{proof}
Assertion (i) is proved by induction using (\ref{3.1}). Clearly, the estimate
holds for $k=0$. In general, we assume the assertion holds for some
$k\in \mathbb{N}$ with $k\Delta \leq T$. Then, using Proposition
\ref{Prop^E Plimi}, we have
\begin{align*}
\big \vert u_{\Delta,N}((k+1)\Delta,x)-u_{\Delta,N}((k+1)\Delta,y)\big \vert
&  =\big \vert \mathbb{\tilde{E}}[u_{\Delta,N}(k\Delta,x+\Delta^{\frac
{1}{\alpha}}\xi^{N})]-\mathbb{\tilde{E}}[u_{\Delta,N}(k\Delta,y+\Delta
^{\frac{1}{\alpha}}\xi^{N})]\big \vert \\
&  \leq \mathbb{\tilde{E}}\big[\big \vert u_{\Delta,N}(k\Delta,x+\Delta
^{\frac{1}{\alpha}}\xi^{N})-u_{\Delta,N}(k\Delta,y+\Delta^{\frac{1}{\alpha}%
}\xi^{N})\big \vert \big]\\
&  \leq C_{\phi}|x-y|.
\end{align*}
By the principle of induction the assertion is true for all $k\in \mathbb{N}$
with $k\Delta \leq T$.

Now we establish the time regularity for $u_{\Delta,N}$ in (ii). Note that
Young's inequality implies that for any $x,y>0$, $xy\leq \frac{1}{2}%
(x^{2}+y^{2})$. For any $\varepsilon>0$, let $x=|x-y|$ and $y=\frac
{1}{\varepsilon}$, then it follows from (i) that
\[
u_{\Delta,N}(k\Delta,x)\leq u_{\Delta,N}(k\Delta,y)+A|x-y|^{2}+B,
\]
where $A=\frac{\varepsilon}{2}C_{\phi}$ and $B=\frac{1}{2\varepsilon}C_{\phi}
$.

We claim that, for any $k\in \mathbb{N}$ such that $k\Delta \leq T$ and
$x,y\in \mathbb{R}$, it holds that
\begin{equation}
u_{\Delta,N}(k\Delta,x)\leq u_{\Delta,N}(0,y)+A|x-y|^{2}+AM_{N}^{2}%
k\Delta^{\frac{2}{\alpha}}+C_{\phi}D_{N}k\Delta^{\frac{1}{\alpha}}+B,
\label{3.3}%
\end{equation}
where $M_{N}^{2}=\mathbb{\tilde{E}}[|\xi^{N}|^{2}]$ and $D_{N}=\mathbb{\tilde
{E}}[|\xi-\xi^{N}|]$. Indeed, (\ref{3.3}) obviously holds for $k=0$. Assume
that for some $k\in \mathbb{N}$ the assertion (\ref{3.3}) holds. Notice that
\begin{equation}
u_{\Delta,N}((k+1)\Delta,x)=\mathbb{\tilde{E}}[u_{\Delta,N}(k\Delta
,x+\Delta^{\frac{1}{\alpha}}\xi^{N})]=\sup_{k_{\pm}\in K_{\pm}}E_{P_{k_{\pm}}
}\big[u_{\Delta,N}(k\Delta,x+\Delta^{\frac{1}{\alpha}}\xi^{N})\big].
\label{3.4}%
\end{equation}
Then, for any $k_{\pm}\in K_{\pm}$,
\begin{equation}%
\begin{split}
E_{P_{k_{\pm}}}[u_{\Delta,N}(k\Delta,x+\Delta^{\frac{1}{\alpha}}\xi^{N})]  &
\leq u_{\Delta,N}(0,y+\Delta^{\frac{1}{\alpha}}E_{P_{k_{\pm}}}[\xi
^{N}])+AM_{N}^{2}k\Delta^{\frac{2}{\alpha}}+C_{\phi}D_{N}k\Delta^{\frac
{1}{\alpha}}\\
&  \text{ \  \ }+B+AE_{P_{k_{\pm}}}\big[\big|x-y+\Delta^{\frac{1}{\alpha}%
}\big(\xi^{N}-E_{P_{k_{\pm}}}[\xi^{N}]\big)\big|^{2}\big].
\end{split}
\end{equation}
Seeing that, $E_{P_{k_{\pm}}}\big[\xi^{N}-E_{P_{k_{\pm}}}[\xi^{N}]\big]=0$
and
\[
E_{P_{k_{\pm}}}\big[\big(\xi^{N}-E_{P_{k_{\pm}}}[\xi^{N}]\big)^{2}%
\big]=E_{P_{k_{\pm}}} \big[(\xi^{N})^{2}\big]-\big(E_{P_{k_{\pm}}}[\xi
^{N}]\big)^{2}\leq \mathbb{\tilde{E}}\big[|\xi^{N}|^{2}\big],
\]
we can deduce that
\begin{equation}
E_{P_{k_{\pm}}}\big[\big|x-y+\Delta^{\frac{1}{\alpha}}\big(\xi^{N}%
-E_{P_{k_{\pm}}}[\xi^{N}]\big)\big|^{2}\big]\leq|x-y|^{2}+M_{N}^{2}%
\Delta^{\frac{2}{\alpha}}.
\end{equation}
Also, since $E_{P_{k_{\pm}}}[\xi]=0$, it follows from (i) that
\begin{equation}
u_{\Delta,N}(0,y+\Delta^{\frac{1}{\alpha}}E_{P_{k_{\pm}}}[\xi^{N}%
])=u_{\Delta,N}(0,y+\Delta^{\frac{1}{\alpha}}E_{P_{k_{\pm}}}[\xi^{N}-\xi])\leq
u_{\Delta,N}(0,y)+C_{\phi}D_{N}\Delta^{\frac{1}{\alpha}}. \label{3.5}%
\end{equation}
Combining (\ref{3.4})-(\ref{3.5}), we obtain that
\[
u_{\Delta,N}((k+1)\Delta,x)\leq u_{\Delta,N}(0,y)+A|x-y|^{2}+AM_{N}%
^{2}(k+1)\Delta^{\frac{2}{\alpha}}+C_{\phi}D_{N}(k+1)\Delta^{\frac{1}{\alpha}%
}+B,
\]
which shows that (\ref{3.3}) also holds for $k+1$. By the principle of
induction our claim is true for all $k\in \mathbb{N}$ such that $k\Delta \leq T$
and $x,y\in \mathbb{R}$. By taking $y=x$ in (\ref{3.3}), we have for any
$\varepsilon>0$,%
\[
u_{\Delta,N}(k\Delta,x)\leq u_{\Delta,N}(0,x)+\frac{\varepsilon}{2}C_{\phi
}M_{N}^{2}k\Delta^{\frac{2}{\alpha}}+C_{\phi}D_{N}k\Delta^{\frac{1}{\alpha}%
}+\frac{1}{2\varepsilon}C_{\phi}.
\]
By minimizing of the right-hand side with respect to $\varepsilon$, we obtain
that%
\begin{align*}
u_{\Delta,N}(k\Delta,x)  &  \leq u_{\Delta,N}(0,x)+C_{\phi}(M_{N}^{2}%
)^{\frac{1}{2}}\Delta^{\frac{2-\alpha}{2\alpha}}(k\Delta)^{\frac{1}{2}%
}+C_{\phi}D_{N}\Delta^{\frac{1-\alpha}{\alpha}}(k\Delta)\\
&  \leq u_{\Delta,N}(0,x)+C_{\phi}\big((M_{N}^{2})^{\frac{1}{2}}\Delta
^{\frac{2-\alpha}{2\alpha}}+D_{N}\Delta^{\frac{1-\alpha}{\alpha}}%
\big)(k\Delta)^{\frac{1}{2}}.
\end{align*}
Similarly, we also have%
\[
u_{\Delta,N}(0,x)\leq u_{\Delta,N}(k\Delta,x)+C_{\phi}\big((M_{N}^{2}%
)^{\frac{1}{2}}\Delta^{\frac{2-\alpha}{2\alpha}}+D_{N}\Delta^{\frac{1-\alpha
}{\alpha}}\big)(k\Delta)^{\frac{1}{2}}.
\]
Combining with Lemmas \ref{truncted_moment estimate}-\ref{truncation error},
we obtain our desired result (ii).
\end{proof}

\begin{lemma}
\label{num-truncated}Suppose that $\phi \in C_{b,Lip}(\mathbb{R})$ and fixed
$N>0$. Then, for any $k\in \mathbb{N}$ such that $k\Delta \leq T$ and
$x\in \mathbb{R}$,
\[
\left \vert u_{\Delta}(k\Delta,x)-u_{\Delta,N}(k\Delta,x)\right \vert \leq
C_{\phi}I_{2,N}N^{1-\alpha}\Delta^{\frac{1-\alpha}{\alpha}}k\Delta,
\]
where $C_{\phi}$ is the Lipschitz constant of $\phi$ and $I_{2,N}$ is given in
Lemma \ref{truncation error}.
\end{lemma}

\begin{proof}
Let $( \xi_{i})_{i\geq1}$ be a sequence of random variables on $(\mathbb{R}%
,C_{Lip}(\mathbb{R}),\mathbb{\tilde{E}})$ such that $\xi_{1}=\xi$, $\xi_{i+1}%
$\ $\overset{d}{=}\xi_{i}$ and $\xi_{i+1}\perp(\xi_{1},\xi_{2},\ldots,\xi
_{i})$ for each $i\in \mathbb{N}$, and let $\xi_{i}^{N}=\xi_{i}\wedge
N\vee(-N)$ for each $i\in \mathbb{N}$. In view of (\ref{2.2}) and (\ref{3.1}),
by using induction method of Theorem 2.1 in \cite{HL2020}, we have for any
$k\in \mathbb{N}$ such that $k\Delta \leq T$ and $x\in \mathbb{R}$,%
\begin{align*}
&  u_{\Delta}(k\Delta,x)=\mathbb{\tilde{E}}[\phi(x+\Delta^{\frac{1}{\alpha}%
}\sum \limits_{i=1}^{k}\xi_{i})],\\
&  u_{\Delta,N}(k\Delta,x)=\mathbb{\tilde{E}}[\phi(x+\Delta^{\frac{1}{\alpha}%
}\sum \limits_{i=1}^{k}\xi_{i}^{N})].
\end{align*}
Then, it follows from the Lipschitz condition of $\phi$\ and Lemma
\ref{truncation error} that%
\[
\left \vert u_{\Delta}(k\Delta,x)-u_{\Delta,N}(k\Delta,x)\right \vert \leq
C_{\phi}\Delta^{\frac{1}{\alpha}}k\mathbb{\tilde{E}}[|\xi_{1}-\xi_{1}%
^{N}|]\leq C_{\phi}I_{2,N}N^{1-\alpha}\Delta^{\frac{1-\alpha}{\alpha}}%
k\Delta,
\]
which we conclude the proof.
\end{proof}

Now we start to prove the regularity results of $u_{\Delta}$.

\begin{proof}
[Proof of Theorem \ref{u_num_regularity}]The space regularity of $u_{\Delta}$
can be proved by induction using (\ref{2.2}). We only focus on the time
regularity of $u_{\Delta}$ and divide its proof into three steps.

Step 1. Consider the special case $\left \vert u_{\Delta}(k\Delta
,\cdot)-u_{\Delta}(0,\cdot)\right \vert $ for any $k\in \mathbb{N}$ such that
$k\Delta \leq T$. Noting that $u_{\Delta,N}(0,x)=u_{\Delta}(0,x)=\phi(x)$, we
have\
\[
\left \vert u_{\Delta}(k\Delta,x)-u_{\Delta}(0,x)\right \vert \leq \left \vert
u_{\Delta}(k\Delta,x)-u_{\Delta,N}(k\Delta,x)\right \vert +\left \vert
u_{\Delta,N}(k\Delta,x)-u_{\Delta,N}(0,x)\right \vert .
\]
In view of Lemmas \ref{u_N regularity} and \ref{num-truncated}, by choosing
$N=\Delta^{-\frac{1}{\alpha}}$, we obtain
\begin{equation}%
\begin{split}
\left \vert u_{\Delta}(k\Delta,x)-u_{\Delta}(0,x)\right \vert  &  \leq C_{\phi
}((I_{1,N})^{\frac{1}{2}}N^{\frac{2-\alpha}{2}}\Delta^{\frac{2-\alpha}%
{2\alpha}}+2I_{2,N}N^{1-\alpha}\Delta^{\frac{1-\alpha}{\alpha}})(k\Delta
)^{\frac{1}{2}}\\
&
\leq C_{\phi}\big((I_{1,\Delta})^{\frac{1}{2}}+2I_{2,\Delta}\big)
(k\Delta)^{\frac{1}{2}},
\end{split}
\label{3.8}%
\end{equation}
where%
\begin{align*}
I_{1,\Delta}  &  =\sup \limits_{k_{\pm}\in K_{\pm}}\bigg \{ \frac{k_{-}+k_{+}%
}{2-\alpha}+2\int_{0}^{1}\frac{|\beta_{1,k_{\pm}}(-\Delta^{-\frac{1}{\alpha}%
}z)|+|\beta_{2,k_{\pm}}(\Delta^{-\frac{1}{\alpha}}z)|}{z^{\alpha-1}}%
dz+|\beta_{1,k_{\pm}}(-\Delta^{-\frac{1}{\alpha}})|+|\beta_{2,k_{\pm}}%
(\Delta^{-\frac{1}{\alpha}})|\bigg \},\\
I_{2,\Delta}  &  =\sup \limits_{k_{\pm}\in K_{\pm}}\bigg \{ \frac{k_{-}+k_{+}%
}{\alpha-1}+\int_{1}^{\infty}\frac{|\beta_{1,k_{\pm}}(-\Delta^{-\frac
{1}{\alpha}}z)|+|\beta_{2,k_{\pm}}(\Delta^{-\frac{1}{\alpha}}z)|}{z^{\alpha}%
}dz+|\beta_{1,k_{\pm}}(-\Delta^{-\frac{1}{\alpha}})|+|\beta_{2,k_{\pm}}%
(\Delta^{-\frac{1}{\alpha}})|\bigg \}.
\end{align*}
In addition, by Assumption (A1), it is easy to obtain that $I_{1,\Delta}$ and
$I_{2,\Delta}$ are finite as $\Delta \rightarrow0$.

Step 2. Let us turn to the case $\left \vert u_{\Delta}(k\Delta,\cdot
)-u_{\Delta}(l\Delta,\cdot)\right \vert $ for any $k,l\in \mathbb{N}$ such that
$(k\vee l)\Delta \leq T$. Without loss of generality, we assume $k\geq l$. Let
$( \xi_{i})_{i=1}^{\infty}$ be a sequence of random variables on
$(\mathbb{R},C_{Lip}(\mathbb{R}),\mathbb{\tilde{E}})$ such that $\xi_{1}=\xi$,
$\xi_{i+1}\overset{d}{=}\xi_{i}$ and $\xi_{i+1}\perp(\xi_{1},\xi_{2}%
,\ldots,\xi_{i})$ for each $i\in \mathbb{N}$. By using induction (\ref{2.2})
and the estimate (\ref{3.8}), it is easy to obtain that for any $k\geq l$ and
$x\in \mathbb{R}$,
\begin{equation}%
\begin{split}
&  \left \vert u_{\Delta}(k\Delta,x)-u_{\Delta}(l\Delta,x)\right \vert \\
&  =\big \vert \mathbb{\tilde{E}}\big [u_{\Delta}\big ((k-l)\Delta
,x+\Delta^{\frac{1}{\alpha}}\sum_{i=1}^{l}\xi_{i}\big )\big ]-\mathbb{\tilde
{E}}\big [u_{\Delta}\big (0,x+\Delta^{\frac{1}{\alpha}}\sum_{i=1}^{l}\xi
_{i}\big )\big ]\big \vert \\
&  \leq \mathbb{\tilde{E}}\big [\big \vert u_{\Delta}\big ((k-l)\Delta
,x+\Delta^{\frac{1}{\alpha}}\sum_{i=1}^{l}\xi_{i}\big )-u_{\Delta
}\big (0,x+\Delta^{\frac{1}{\alpha}}\sum_{i=1}^{l}\xi_{i}%
\big )\big \vert \big ]\\
&  \leq C_{\phi}((I_{1,\Delta})^{\frac{1}{2}}+2I_{2,\Delta})((k-l)\Delta
)^{\frac{1}{2}}.
\end{split}
\label{3.9}%
\end{equation}

Step 3. In general, for $s,t\in \lbrack0,T]$, let $\delta_{s},\delta_{t}%
\in \lbrack0,\Delta)$ such that $s-\delta_{s}$ and $t-\delta_{t}$ are in the
grid points $\{k\Delta:k\in \mathbb{N}\}$. Then, from (\ref{3.9}), we have%
\begin{align*}
u_{\Delta}(t,x)=u_{\Delta}(t-\delta_{t},x)  &  \leq u_{\Delta}(s-\delta
_{s},x)+C_{\phi}((I_{1,\Delta})^{\frac{1}{2}}+2I_{2,\Delta})|t-s-\delta
_{t}+\delta_{s}|^{\frac{1}{2}}\\
&  \leq u_{\Delta}(s,x)+C_{\phi}((I_{1,\Delta})^{\frac{1}{2}}+2I_{2,\Delta
})(|t-s|^{\frac{1}{2}}+\Delta^{\frac{1}{2}}).
\end{align*}
Similarly one proves that%
\[
u_{\Delta}(s,x)\leq u_{\Delta}(t,x)+C_{\phi}((I_{1,\Delta})^{\frac{1}{2}%
}+2I_{2,\Delta})(|t-s|^{\frac{1}{2}}+\Delta^{\frac{1}{2}}),
\]
and this yields (ii).
\end{proof}

\subsection{The monotone approximation scheme}

In this section, we first rewrite the recursive approximation (\ref{2.2}) as a
monotone scheme, and then derive its consistency error estimates and
comparison result.

For $\Delta \in(0,1)$, based on (\ref{2.2}),\ we introduce the monotone
approximation scheme as
\begin{equation}
\left \{
\begin{array}
[c]{l}%
S(\Delta,x,u_{\Delta}(t,x),u_{\Delta}(t-\Delta,\cdot))=0,\text{\ }
(t,x)\in \lbrack \Delta,T]\times \mathbb{R},\\
u_{\Delta}(t,x)=\phi(x),\text{\ }(t,x)\in \lbrack0,\Delta)\times \mathbb{R},
\end{array}
\right.  \label{4.1}%
\end{equation}
where $S:(0,1)\times \mathbb{R}\times \mathbb{R}\times C_{b}
(\mathbb{R)\rightarrow R}$ is defined by
\begin{equation}
S(\Delta,x,p,v)=\frac{p-\mathbb{\tilde{E}}[v(x+\Delta^{\frac{1}{\alpha}}\xi
)]}{\Delta}. \label{4.1.operator}%
\end{equation}

For a function $f$ defined on $[0,T]\times \mathbb{R}$, introduce its norm
$|f|_{0}:=\sup \limits_{[0,T]\times \mathbb{R}}|f(t,x)|$. We now give key
properties of the approximation scheme (\ref{4.1}).

\begin{proposition}
\label{prop}Suppose that $S(\Delta,x,p,v)$ is given in (\ref{4.1.operator}).
Then, the following properties hold:

\begin{description}
\item[(i) (Monotonicity)] For any $c_{1},c_{2}\in \mathbb{R}$ and any function
$u\in C_{b}(\mathbb{R)}$ with $u\leq v,$%
\[
S(\Delta,x,p+c_{1},u+c_{2})\geq S(\Delta,x,p,v)+\frac{c_{1}-c_{2}}{\Delta};
\]

\item[(ii) (Concavity)] For any $\lambda \in[0,1]$, $p_{1},p_{2}\in \mathbb{R}$,
and $v_{1},v_{2}\in C_{b}(\mathbb{R)}$, then $S(\Delta,x,p,v)$ is concave in
$(p,v)$, that is,
\begin{align*}
& S\left(  \Delta,x,\lambda p_{1}+(1-\lambda)p_{2},\lambda v_{1}(\cdot)+(1-\lambda)v_{2}(\cdot)\right)
\\
& \geq \lambda S\left(  \Delta,x,p_{1},v_{1}(\cdot)\right)  +(1-\lambda)S\left(
\Delta,x,p_{2},v_{2}(\cdot)\right)  ;
\end{align*}

\item[(iii) (Consistency)] For any $\omega \in C_{b}^{\infty}([\Delta
,T]\times \mathbb{R)}$, then
\[%
\begin{array}
[c]{l}%
\big \vert \partial_{t}\omega(t,x)-\sup \limits_{k_{\pm}\in K_{\pm}%
}\big \{ \int_{\mathbb{R}}\delta_{z}\omega(t,x)F_{k_{\pm}}(dz)\big \}-S(\Delta
,x,\omega(t,x),\omega(t-\Delta,\cdot))\big \vert \\
\leq(1+\mathbb{\tilde{E}}\left[  |\xi|\right]  )(|\partial_{t}^{2}\omega
|_{0}\Delta+|\partial_{t}D_{x}\omega|_{0}\Delta^{\frac{1}{\alpha}}%
)+R^{0}|D_{x}^{2}\omega|_{0}\Delta^{\frac{2-\alpha}{\alpha}}+|D_{x}^{2}%
\omega|_{0}R_{\Delta}^{1}+\left \vert D_{x}\omega \right \vert _{0}R_{\Delta}%
^{2},
\end{array}
\]
where
\[%
\begin{array}
[c]{rl}%
R^{0}= & \displaystyle \sup_{k_{\pm}\in K_{\pm}}\bigg \{|\beta_{1,k_{\pm}%
}(-1)|+|\beta_{2,k_{\pm}}(1)|+\int_{0}^{1}\big [|\alpha \beta_{1,k_{\pm}%
}(-z)+\beta_{1,k_{\pm}}^{\prime}(-z)z|\\
& \displaystyle+|\alpha \beta_{2,k_{\pm}}(z)-\beta_{2,k_{\pm}}^{\prime
}(z)z|\big ]z^{1-\alpha}dz\bigg \},\\
R_{\Delta}^{1}= & 5\displaystyle \sup_{k_{\pm}\in K_{\pm}}\bigg \{ \int
_{0}^{1}\left[  |\beta_{1,k_{\pm}}(-\Delta^{-\frac{1}{\alpha}}z)|+|\beta
_{2,k_{\pm}}(\Delta^{-\frac{1}{\alpha}}z)|\right]  z^{1-\alpha}dz\bigg \},\\
R_{\Delta}^{2}= & \displaystyle4\sup_{k_{\pm}\in K_{\pm}}\bigg \{|\beta
_{1,k_{\pm}}(-\Delta^{-\frac{1}{\alpha}})|+|\beta_{2,k_{\pm}}(\Delta
^{-\frac{1}{\alpha}})|\\
& \displaystyle+\int_{1}^{\infty}\left[  |\beta_{1,k_{\pm}}(-\Delta^{-\frac
{1}{\alpha}}z)|+|\beta_{2,k_{\pm}}(\Delta^{-\frac{1}{\alpha}}z)|\right]
z^{-\alpha}dz\bigg \}.
\end{array}
\]

\end{description}
\end{proposition}

\begin{proof}
Parts (i)-(ii) are immediate, so we only prove (iii). To this end, we split
the consistency error into two parts. Specifically, for $(t,x)\in
[\Delta,T]\times \mathbb{R}$,%
\[%
\begin{array}
[c]{l}%
\Big \vert \partial_{t}\omega(t,x)-\sup \limits_{k_{\pm}\in K_{\pm}%
}\big \{ \int_{\mathbb{R}}\delta_{z}\omega(t,x)F_{k_{\pm}}(dz)\big \}-S(\Delta
,x,\omega(t,x),\omega(t-\Delta,\cdot))\Big \vert \\
\leq \Delta^{-1}\Big \vert \mathbb{\tilde{E}}[\omega(t-\Delta,x+\Delta
^{\frac{1}{\alpha}}\xi)]-\mathbb{\tilde{E}}[\omega(t,x+\Delta^{\frac{1}%
{\alpha}} \xi)-D_{x}\omega(t,x)\Delta^{\frac{1}{\alpha}}\xi]+\partial
_{t}\omega(t,x)\Delta \Big \vert \\
\text{ \ }+\Delta^{-1}\Big \vert \mathbb{\tilde{E}}[\delta_{\Delta^{1/\alpha
}\xi}\omega(t,x)]-\sup \limits_{k_{\pm}\in K_{\pm}}\big \{ \int_{\mathbb{R}%
}\delta_{z}\omega(t,x)F_{k_{\pm}}(dz)\big \} \Delta \Big \vert:=I+II.
\end{array}
\]
Applying Taylor's expansion (twice) yields that
\begin{equation}
\omega(t,x+\Delta^{\frac{1}{\alpha}}\xi)=\omega(t-\Delta,x+\Delta^{\frac
{1}{\alpha}}\xi)+\int_{t-\Delta}^{t}\partial_{t}\omega(s,x)ds+\int_{t-\Delta
}^{t}\int_{x}^{x+\Delta^{1/\alpha}\xi}\partial_{t}D_{x}\omega(s,y)dyds.
\label{4.2}%
\end{equation}
Since $\mathbb{\tilde{E}}[\xi]=\mathbb{\tilde{E}}[-\xi]=0$, then (\ref{4.2})
and the mean value theorem give
\begin{equation}%
\begin{split}
I  &  \leq \Delta^{-1}\int_{t-\Delta}^{t}\left \vert \partial_{t}\omega
(t,x)-\partial_{t}\omega(s,x)\right \vert ds+\Delta^{-1}\mathbb{\tilde{E}%
}\bigg[\bigg \vert \int_{t-\Delta}^{t}\int_{x}^{x+\Delta^{1/\alpha}\xi
}\partial_{t}D_{x}\omega(s,y)dyds\bigg \vert \bigg]\\
&  \leq \frac{1}{2}|\partial_{t}^{2}\omega|_{0}\Delta+\mathbb{\tilde{E}}%
[|\xi|]|\partial_{t}D_{x}\omega|_{0}\Delta^{\frac{1}{\alpha}}.
\end{split}
\label{4.4}%
\end{equation}
For the part $II$, by changing variables, we get%
\begin{align*}
II  &  \leq \sup \limits_{k_{\pm}\in K_{\pm}}\bigg \{ \bigg \vert \int
_{\mathbb{-\infty}}^{0}\delta_{z}\omega(t,x)[-\beta_{1,k_{\pm}}^{\prime
}(\Delta^{-\frac{1}{\alpha}}z)\Delta^{-\frac{1}{\alpha}}z+\alpha
\beta_{1,k_{\pm}}(\Delta^{-\frac{1}{\alpha}}z)]|z|^{-\alpha-1}dz\\
&  \  \  \  \ +\int_{0}^{\infty}\delta_{z}\omega(t,x)[-\beta_{2,k_{\pm}}^{\prime
}(\Delta^{-\frac{1}{\alpha}}z)\Delta^{-\frac{1}{\alpha}}z+\alpha
\beta_{2,k_{\pm}}(\Delta^{-\frac{1}{\alpha}}z)]z^{-\alpha-1}%
dz\bigg \vert \bigg \}.
\end{align*}
We only consider the integral above along the positive half-line, and
similarly for the integral along the negative half-line. For simplicity, we
set%
\[%
\begin{array}
[c]{l}%
\rho=\delta_{z}\omega(t,x)[-\beta_{2,k_{\pm}}^{\prime}(\Delta^{-\frac
{1}{\alpha}}z)\Delta^{-\frac{1}{\alpha}}z+\alpha \beta_{2,k_{\pm}}%
(\Delta^{-\frac{1}{\alpha}}z)]z^{-\alpha-1},\\
\int_{0}^{\infty}\rho dz=\int_{1}^{\infty}\rho dz+\int_{\Delta^{1/\alpha}}%
^{1}\rho dz+\int_{0}^{\Delta^{1/\alpha}}\rho dz:=J_{1}+J_{2}+J_{3}.
\end{array}
\]
Using integration by parts, we have for any $k_{\pm}\in K_{\pm}$,
\begin{align*}
|J_{1}|  &  =\bigg \vert \delta_{1}\omega(t,x)\beta_{2,k_{\pm}}(\Delta
^{-\frac{1}{\alpha}})\\
&  \text{ \  \  \ }+\int_{1}^{\infty}\beta_{2,k_{\pm}}(\Delta^{-\frac{1}{\alpha
}}z)[D_{x}\omega(t,x+z)-D_{x}\omega(t,x)]z^{-\alpha}dz\bigg \vert \\
&  \leq2\left \vert D_{x}\omega \right \vert _{0}\bigg (|\beta_{2,k_{\pm}}%
(\Delta^{-\frac{1}{\alpha}})|+\int_{1}^{\infty}|\beta_{2,k_{\pm}}%
(\Delta^{-\frac{1}{\alpha}}z)|z^{-\alpha}dz\bigg),
\end{align*}
where we have used the fact that for $\theta \in(0,1)$
\[
|\delta_{1}\omega(t,x)|=|D_{x}\omega(t,x+\theta)-D_{x}\omega(t,x)|\leq
2|D_{x}\omega|_{0}.
\]
Notice that for any $k_{\pm}\in K_{\pm}$,%
\begin{align*}
|J_{2}|  &  \leq \bigg \vert \int_{\Delta^{1/\alpha}}^{1}\alpha \delta_{z}%
\omega(t,x)\beta_{2,k_{\pm}}(\Delta^{-\frac{1}{\alpha}}z)z^{-\alpha
-1}dz\bigg \vert \\
&  \text{ \  \  \ }+\bigg \vert \int_{\Delta^{1/\alpha}}^{1}\delta_{z}%
\omega(t,x)[-\beta_{2,k_{\pm}}^{\prime}(\Delta^{-\frac{1}{\alpha}}%
z)\Delta^{-\frac{1}{\alpha}}z]z^{-\alpha-1}dz\bigg \vert.
\end{align*}
By means of integration by parts and the mean value theorem, we obtain%
\begin{align*}
&  \bigg \vert \int_{\Delta^{1/\alpha}}^{1}\delta_{z}\omega(t,x)[-\beta
_{2,k_{\pm}}^{\prime}(\Delta^{-\frac{1}{\alpha}}z)\Delta^{-\frac{1}{\alpha}%
}z]z^{-\alpha-1}dz\bigg \vert \\
&  =\bigg \vert \delta_{\Delta^{1/\alpha}}\omega(t,x)\beta_{2,k_{\pm}%
}(1)\Delta^{-1}-\delta_{1}\omega(t,x)\beta_{2,k_{\pm}}(\Delta^{-\frac
{1}{\alpha}})\\
&  \text{ \  \  \ }+\int_{\Delta^{1/\alpha}}^{1}\beta_{2,k_{\pm}}(\Delta
^{-\frac{1}{\alpha}}z)[D_{x}\omega(t,x+z)-D_{x}\omega(t,x)]z^{-\alpha}dz\\
&  \text{ \  \  \ }-\alpha \int_{\Delta^{1/\alpha}}^{1}\delta_{z}\omega
(t,x)\beta_{2,k_{\pm}}(\Delta^{-\frac{1}{\alpha}}z)z^{-\alpha-1}%
dz\bigg \vert \\
&  \leq|D_{x}^{2}\omega|_{0}|\beta_{2,k_{\pm}}(1)|\Delta^{\frac{2-\alpha
}{\alpha}}+2|D_{x}\omega|_{0}|\beta_{2,k_{\pm}}(\Delta^{-\frac{1}{\alpha}})|\\
&  \text{ \  \  \ }+(\alpha+1)|D_{x}^{2}\omega|_{0}\int_{0}^{1}|\beta_{2,k_{\pm
}}(\Delta^{-\frac{1}{\alpha}}z)|z^{1-\alpha}dz,
\end{align*}
by using the fact that for $\theta \in(0,1)$
\[
|\delta_{z}\omega(t,x)|=\frac{1}{2}|D_{x}^{2}\omega(t,x+\theta z)z^{2}%
|\leq|D_{x}^{2}\omega|_{0}z^{2},
\]
and similarly,
\[
\bigg \vert \int_{\Delta^{1/\alpha}}^{1}\alpha \delta_{z}\omega(t,x)\beta
_{2,k_{\pm}}(\Delta^{-\frac{1}{\alpha}}z)z^{-\alpha-1}dz\bigg \vert \leq
\alpha|D_{x}^{2}\omega|_{0}\int_{0}^{1}|\beta_{2,k_{\pm}}(\Delta^{-\frac
{1}{\alpha}}z)|z^{1-\alpha}dz.
\]
In the same way, we can also obtain
\begin{align*}
|J_{3}|  &  \leq|D_{x}^{2}\omega|_{0}\int_{0}^{\Delta^{1/\alpha}%
}\big \vert \alpha \beta_{2,k_{\pm}}(\Delta^{-\frac{1}{\alpha}}z)-\beta
_{2,k_{\pm}}^{\prime}(\Delta^{-\frac{1}{\alpha}}z)\Delta^{-\frac{1}{\alpha}%
}z\big \vert z^{1-\alpha}dz\\
&  =|D_{x}^{2}\omega|_{0}\Delta^{\frac{2-\alpha}{\alpha}}\int_{0}%
^{1}\big \vert \alpha \beta_{2,k_{\pm}}(z)-\beta_{2,k_{\pm}}^{\prime
}(z)z\big \vert z^{1-\alpha}dz.
\end{align*}
Together with $J_{1},J_{2}$ and $J_{3}$, we conclude that
\begin{equation}%
\begin{split}
II  &  \leq4\left \vert D_{x}\omega \right \vert _{0}\sup_{k_{\pm}\in K_{\pm}%
}\bigg \{|\beta_{1,k_{\pm}}(-\Delta^{-\frac{1}{\alpha}})|+|\beta_{2,k_{\pm}%
}(\Delta^{-\frac{1}{\alpha}})|\\
&  \text{ \  \ }+\int_{1}^{\infty}[|\beta_{1,k_{\pm}}(-\Delta^{-\frac{1}%
{\alpha}}z)|+|\beta_{2,k_{\pm}}(\Delta^{-\frac{1}{\alpha}}z)|]z^{-\alpha
}dz\bigg \} \\
\text{ \ }  &  \text{ \  \ }+(1+2\alpha)|D_{x}^{2}\omega|_{0}\sup_{k_{\pm}\in
K_{\pm}}\bigg \{ \int_{0}^{1}[|\beta_{1,k_{\pm}}(-\Delta^{-\frac{1}{\alpha}%
}z)|+|\beta_{2,k_{\pm}}(\Delta^{-\frac{1}{\alpha}}z)|]z^{1-\alpha}dz\bigg \} \\
&  \text{ \  \ }+\Delta^{\frac{2-\alpha}{\alpha}}|D_{x}^{2}\omega|_{0}%
\sup_{k_{\pm}\in K_{\pm}}\bigg \{|\beta_{1,k_{\pm}}(-1)|+|\beta_{2,k_{\pm}%
}(1)|\\
&  \text{ \  \ }+\int_{0}^{1}[|\alpha \beta_{1,k_{\pm}}(-z)+\beta_{1,k_{\pm}%
}^{\prime}(-z)z|+|\alpha \beta_{2,k_{\pm}}(z)-\beta_{2,k_{\pm}}^{\prime
}(z)z|]z^{1-\alpha}dz\bigg \}.
\end{split}
\label{4.5}%
\end{equation}
Consequently, the desired conclusion follows from \eqref{4.4} and \eqref{4.5}.
\end{proof}

From Proposition \ref{prop} (i) we can derive the following comparison result
for the scheme (\ref{4.1}), which will be used throughout this paper.

\begin{lemma}
\label{comparison theorem}Suppose that $\underline{v},\bar{v}\in
C_{b}([0,T]\times \mathbb{R)}$ satisfy%
\begin{align*}
S(\Delta,x,\underline{v}(t,x),\underline{v}(t-\Delta,\cdot))  &  \leq
h_{1}\text{ \ in }(\Delta,T]\times \mathbb{R}\text{,}\\
S(\Delta,x,\bar{v}(t,x),\bar{v}(t-\Delta,\cdot))  &  \geq h_{2}\text{ \ in
}(\Delta,T]\times \mathbb{R}\text{,}%
\end{align*}
where $h_{1},h_{2}\in C_{b}((\Delta,T]\times \mathbb{R})$. Then
\[
\underline{v}-\bar{v}\leq \sup_{(t,x)\in \lbrack0,\Delta]\times \mathbb{R}%
}(\underline{v}-\bar{v})^{+}+t\sup_{(t,x)\in(\Delta,T]\times \mathbb{R}}%
(h_{1}-h_{2})^{+}\text{.}%
\]

\end{lemma}

\begin{proof}
The basic idea of the proof comes from Lemma 3.2 in \cite{BJ2007}. For
reader's convenience, we shall give the sketch of the proof. We first note
that it suffices to prove the lemma in the case
\[
\underline{v}\leq \bar{v}\text{ \ in }[0,\Delta]\times \mathbb{R}\text{,
\ }h_{1}\leq h_{2}\  \  \text{in }(\Delta,T]\times \mathbb{R}.
\]
The general case follows from this after seeing that the monotonicity property
in Proposition \ref{prop} (i)
\[
\omega:=\bar{v}+\sup_{(t,x)\in \lbrack0,\Delta]\times \mathbb{R}}(\underline
{v}-\bar{v})^{+}+t\sup_{(t,x)\in(\Delta,T]\times \mathbb{R}}(h_{1}-h_{2})^{+}%
\]
satisfies
\[
S(\Delta,x,\omega(t,x),\omega(t-\Delta,\cdot))\geq S(\Delta,x,\bar
{v}(t,x),\bar{v}(t-\Delta,\cdot))+\sup_{(t,x)\in(\Delta,T]\times \mathbb{R}%
}(h_{1}-h_{2})^{+}\geq h_{1},
\]
for $(t,x)\in(\Delta,T]\times \mathbb{R}$, and $\underline{v}\leq \omega$ in
$[0,\Delta]\times \mathbb{R}$.

For $c\geq0$, let $\psi_{c}(t):=ct$\ and $g(c):=\sup_{(t,x)\in \lbrack
0,T]\times \mathbb{R}}\{ \underline{v}-\bar{v}-\psi_{c}\}$.\ Next, we have to
prove that $g(0)\leq0$ and we argue by contradiction assuming $g(0)>0$. From
the continuity of $g$, we can find some $c>0$ such that $g(c)>0$. For such
$c$, take a sequence $\{(t_{n},x_{n})\}_{n\geq1}\subset$ $[0,T]\times
\mathbb{R}$ such that
\[
\delta_{n}:=g(c)-(\underline{v}-\bar{v}-\psi_{c})(t_{n},x_{n})\rightarrow
0\text{, as }n\rightarrow \infty \text{.}%
\]
Since $\underline{v}-\bar{v}-\psi_{c}\leq0$ in $[0,\Delta]\times \mathbb{R}$
and $g(c)>0$, we assert that $t_{n}>\Delta$ for sufficiently large $n$. For
such $n$, applying Proposition \ref{prop} (i) (twice) we can deduce
\begin{align*}
h(t_{n},x_{n})  &  \geq S(\Delta,x,\underline{v}(t,x),\underline{v}%
(t-\Delta,\cdot))\\
&  \geq S(\Delta,x,\bar{v}(t_{n},x_{n})+\psi_{c}(t_{n})+g(c)-\delta_{n}%
,\bar{v}(t_{n}-\Delta,\cdot)+\psi_{c}(t_{n}-\Delta)+g(c))\\
&  \geq S(\Delta,x,\bar{v}(t_{n},x_{n}),\bar{v}(t_{n}-\Delta,\cdot))+(\psi
_{c}(t_{n})-\psi_{c}(t_{n}-\Delta)-\delta_{n})\Delta^{-1}\\
&  \geq h_{2}(t_{n},x_{n})+c-\delta_{n}\Delta^{-1}.
\end{align*}
Since $h_{1}\leq h_{2}$ in $(\Delta,T]\times \mathbb{R}$, this yields that
$c-\delta_{n}\Delta^{-1}\leq0$. By letting $n\rightarrow \infty$, we obtain
$c\leq0$, which is a contradiction.
\end{proof}

\subsection{Convergence rate of the monotone approximation scheme}

In this subsection, we shall prove the convergence rate of the monotone
approximation scheme $u_{\Delta}$ in Theorem \ref{main theorem 2}. The
convergence of the approximate solution $u_{\Delta}$ to the viscosity solution
$u$ follows from a nonlocal extension of the Barles-Souganidis half-relaxed
limits method \cite{BS1991}.

We start from the first time interval $[0,\Delta]\times \mathbb{R}$.

\begin{lemma}
\label{first interval estimate}Suppose that $\phi \in C_{b,Lip}(\mathbb{R})$.
Then, for $(t,x)\in \lbrack0,\Delta]\times \mathbb{R}$,
\begin{equation}
|u(t,x)-u_{\Delta}(t,x)|\leq C_{\phi}(M_{X}^{1}+M_{\xi}^{1})\Delta^{\frac
{1}{\alpha}}, \label{5.1}%
\end{equation}
where $C_{\phi}$ is the Lipschitz constant of $\phi$, $M_{\xi}^{1}%
:=\mathbb{\tilde{E}}[|\xi|]$ and $M_{X}^{1}:=\mathbb{\hat{E}}[|X_{1}|]$.
\end{lemma}

\begin{proof}
Clearly, (\ref{5.1}) holds in $(t,x)\in \lbrack0,\Delta)\times \mathbb{R}$,
since
\[
u(0,x)=u_{\Delta}(t,x)=\phi(x)\text{, \ }(t,x)\in \lbrack0,\Delta
)\times \mathbb{R}.
\]
For $t=\Delta$, from Lemma \ref{DPP} and (\ref{2.2}), we obtain that
\begin{align*}
|u(\Delta,x)-u_{\Delta}(\Delta,x)|  &  \leq|u(\Delta,x)-u(0,x)|+|u_{\Delta
}(0,x)-u_{\Delta}(\Delta,x)|\\
&  \leq \mathbb{\hat{E}}[|\phi(x+X_{\Delta})-\phi(x)|]+\mathbb{\tilde{E}}%
[|\phi(x)-\phi(x+\Delta^{\frac{1}{\alpha}}\xi)|]\\
&  \leq C_{\phi}(\mathbb{\hat{E}}[|X_{1}|]+\mathbb{\tilde{E}}[|\xi
|])\Delta^{\frac{1}{\alpha}},
\end{align*}
which implies the desired result.
\end{proof}

\subsubsection{Lower bound for the error of approximation scheme}

In order to obtain the lower bound for the approximation scheme, we follow
Krylov's regularization results \cite{Krylov1997,Krylov1999,Krylov2000} (see
also \cite{BJ2002,BJ2005} for analogous results under PDE arguments). For
$\varepsilon \in(0,1)$, we first extend (\ref{u_PIDE}) to the domain
$[0,T+\varepsilon^{2}]\times \mathbb{R}$ and still denote as $u$. For
$(t,x)\in \lbrack0,T]\times \mathbb{R}$, we define the mollification of $u$ by
\[
u^{\varepsilon}(t,x)=u\ast \zeta_{\varepsilon}(t,x)=\int_{-\varepsilon^{2}%
<\tau<0}\int_{|e|<\varepsilon}u(t-\tau,x-e)\zeta_{\varepsilon}(\tau
,e)ded\tau.
\]
In view of Lemma \ref{u_regularity}, the standard properties of mollifiers
indicate that%
\begin{equation}%
\begin{array}
[c]{l}%
|u-u^{\varepsilon}|_{0}\leq C_{\phi,\mathcal{K}}(\varepsilon+\varepsilon
^{\frac{2}{\alpha}})\leq2C_{\phi,\mathcal{K}}\varepsilon,\\
|\partial_{t}^{l}D_{x}^{k}u^{\varepsilon}|_{0}\leq C_{\phi,\mathcal{K}%
}M_{\zeta}(\varepsilon+\varepsilon^{\frac{2}{\alpha}})\varepsilon^{-2l-k}%
\leq2C_{\phi,\mathcal{K}}M_{\zeta}\varepsilon^{1-2l-k}\text{ \ for }k+l\geq1,
\end{array}
\label{5.3}%
\end{equation}
where
\[
M_{\zeta}:=\max \limits_{k+l\geq1}\int_{-1<t<0}\int_{|x|<1}|\partial_{t}%
^{l}D_{x}^{k}\zeta(t,x)|dxdt<\infty.
\]

We obtain the following lower bound.

\begin{lemma}
\label{lower bound}Suppose that (A1)-(A3) hold and $\phi \in C_{b,Lip}%
(\mathbb{R})$. Then, for $(t,x)\in \lbrack0,T]\times \mathbb{R}$,%
\[
u_{\Delta}(t,x)\leq u(t,x)+L_{0}\Delta^{\Gamma(\alpha,q)},
\]
where $\Gamma(\alpha,q)=\min \{ \frac{1}{4},\frac{2-\alpha}{2\alpha},\frac
{q}{2}\}$ and $L_{0}$ is a constant depending on $C_{\phi},C_{\phi
,\mathcal{K}},M_{X}^{1},M_{\xi}^{1},M_{\zeta},M$ given in (\ref{L}).
\end{lemma}

\begin{proof}
Step 1. Notice that $u(t-\tau,x-e)$ is a viscosity solution of (\ref{u_PIDE})
in $[0,T]\times \mathbb{R}$ for any $(\tau,e)\in(-\varepsilon^{2},0)\times
B(0,\varepsilon)$. Multiplying it by $\zeta_{\varepsilon}(\tau,e)$ and
integrating it with respect to $(\tau,e)$, from the concavity of
(\ref{u_PIDE}) with respect to the nonlocal term, we can derive that
$u^{\varepsilon}(t,x)$ is a supersolution of (\ref{u_PIDE}) in $(0,T]\times
\mathbb{R}$, that is, for $(t,x)\in(0,T]\times \mathbb{R}$,%
\begin{equation}
\partial_{t}u^{\varepsilon}(t,x)-\sup \limits_{k_{\pm}\in K_{\pm}}\left \{
\int_{\mathbb{R}}\delta_{z}u^{\varepsilon}(t,x)F_{k_{\pm}}(dz)\right \}  \geq0.
\label{5.2}%
\end{equation}

Step 2. Since $u^{\varepsilon}\in C_{b}^{\infty}([0,T]\times \mathbb{R)}$,
together with the consistency property in Proposition \ref{prop} (iii) and
(\ref{5.2}), using (\ref{5.3}), we can deduce that
\begin{equation}%
\begin{split}
&  S(\Delta,x,u^{\varepsilon}(t,x),u^{\varepsilon}(t-\Delta,\cdot))\\
&  \geq-2C_{\phi,\mathcal{K}}M_{\zeta}[(1+M_{\xi}^{1})(\varepsilon^{-3}%
\Delta+\varepsilon^{-2}\Delta^{\frac{1}{\alpha}})+\varepsilon^{-1}%
\Delta^{\frac{2-\alpha}{\alpha}}R^{0}+\varepsilon^{-1}R_{\Delta}^{1}%
+R_{\Delta}^{2}]\\
&  =:-2C_{\phi,\mathcal{K}}M_{\zeta}C(\varepsilon,\Delta).
\end{split}
\label{5.4}%
\end{equation}
Applying comparison principle in Lemma \ref{comparison theorem} to $u_{\Delta
}$ and $u^{\varepsilon}$,\ by (\ref{4.1}) and (\ref{5.4}), we have for
$(t,x)\in \lbrack0,T]\times \mathbb{R}$,
\begin{equation}
u_{\Delta}-u^{\varepsilon}\leq \sup_{(t,x)\in \lbrack0,\Delta]\times \mathbb{R}%
}(u_{\Delta}-u^{\varepsilon})^{+}+2TC_{\phi,\mathcal{K}}M_{\zeta}%
C(\varepsilon,\Delta). \label{5.5}%
\end{equation}

Step 3. In view of Lemma \ref{first interval estimate} and (\ref{5.5}), we
obtain that
\begin{align*}
u_{\Delta}-u  &  =(u_{\Delta}-u^{\varepsilon})+(u^{\varepsilon}-u)\\
&  \leq \sup_{(t,x)\in \lbrack0,\Delta]\times \mathbb{R}}(u_{\Delta}%
-u)^{+}+|u-u^{\varepsilon}|+2TC_{\phi,\mathcal{K}}M_{\zeta}C(\varepsilon
,\Delta)+2C_{\phi,\mathcal{K}}\varepsilon \\
&  \leq C_{\phi}(M_{X}^{1}+M_{\xi}^{1})\Delta^{\frac{1}{\alpha}}%
+2TC_{\phi,\mathcal{K}}M_{\zeta}C(\varepsilon,\Delta)+4C_{\phi,\mathcal{K}%
}\varepsilon.
\end{align*}
Assumptions (A1)-(A3) indicate that $R^{0}\leq4M$, $R_{\Delta}^{1}%
\leq10C\Delta^{q}$, and $R_{\Delta}^{2}\leq16C\Delta^{q}$. When $\alpha
\in(1,\frac{4}{3}]$ and $q\in \lbrack \frac{1}{2},\infty)$, by choosing
$\varepsilon=\Delta^{\frac{1}{4}}$, we have $u_{\Delta}-u\leq L_{0}%
\Delta^{\frac{1}{4}}$, where%
\begin{equation}
L_{0}:=C_{\phi}(M_{X}^{1}+M_{\xi}^{1})+4C_{\phi,\mathcal{K}}+2TC_{\phi
,\mathcal{K}}M_{\zeta}[2(1+M_{\xi}^{1})+4M+26C]\text{;} \label{L}%
\end{equation}
when $\alpha \in(1,\frac{4}{3}]$ and $q\in \lbrack0,\frac{1}{2})$, by choosing
$\varepsilon=\Delta^{\frac{q}{2}}$, we have $u_{\Delta}-u\leq L_{0}%
\Delta^{\frac{q}{2}}$; when $\alpha \in(\frac{4}{3},2)$ and $q\in \lbrack
\frac{2-\alpha}{\alpha},\infty)$, by letting $\varepsilon=\Delta
^{\frac{2-\alpha}{2\alpha}}$, we get $u_{\Delta}-u\leq L_{0}\Delta
^{\frac{2-\alpha}{2\alpha}}$; when $\alpha \in(\frac{4}{3},2)$ and
$q\in(0,\frac{2-\alpha}{\alpha})$, by letting $\varepsilon=\Delta^{\frac{q}%
{2}}$, we get $u_{\Delta}-u\leq L_{0}\Delta^{\frac{q}{2}}$. To sum up, we
conclude that%
\[
u_{\Delta}-u\leq L_{0}\Delta^{\Gamma(\alpha,q)},
\]
where $\Gamma(\alpha,q)=\min \{ \frac{1}{4},\frac{2-\alpha}{2\alpha},\frac
{q}{2}\}$. This leads to the desired result.
\end{proof}

\subsubsection{Upper bound for the error of approximation schemes}

To obtain an upper bound for the error of approximation scheme, we are not
able to construct approximate smooth subsolutions of (\ref{u_PIDE}) due to the
concavity of (\ref{u_PIDE}). Instead, we interchange the roles of PIDE
(\ref{u_PIDE}) and the approximation scheme (\ref{4.1}). For $\varepsilon
\in(0,1)$, we extend (\ref{4.1}) to the domain $[0,T+\varepsilon^{2}%
]\times \mathbb{R}$ and still denote as $u_{\Delta}$. For $(t,x)\in
\lbrack0,T]\times \mathbb{R}$, we define the mollification of $u$ by
\[
u_{\Delta}^{\varepsilon}(t,x)=u_{\Delta}\ast \zeta_{\varepsilon}(t,x)=\int
_{-\varepsilon^{2}<\tau<0}\int_{|e|<\varepsilon}u_{\Delta}(t-\tau
,x-e)\zeta_{\varepsilon}(\tau,e)d\tau de.
\]
In view of Theorem \ref{u_num_regularity}, the standard properties of
mollifiers indicate that%
\begin{equation}%
\begin{array}
[c]{l}%
|u_{\Delta}-u_{\Delta}^{\varepsilon}|_{0}\leq C_{\phi}(1+I_{\Delta
})(\varepsilon+\Delta^{\frac{1}{2}}),\\
|\partial_{t}^{l}D_{x}^{k}u_{\Delta}^{\varepsilon}|_{0}\leq C_{\phi}M_{\zeta
}(1+I_{\Delta})(\varepsilon+\Delta^{\frac{1}{2}})\varepsilon^{-2l-k}\text{
\ for }k+l\geq1.
\end{array}
\label{5.6}%
\end{equation}

We obtain the following upper bound.

\begin{lemma}
\label{upper bound}Suppose that (A1)-(A3) hold and $\phi \in C_{b,Lip}%
(\mathbb{R})$. Then, for $(t,x)\in \lbrack0,T]\times \mathbb{R}$,
\[
u(t,x)\leq u_{\Delta}(t,x)+U_{0}\Delta^{\Gamma(\alpha,q)},
\]
where $\Gamma(\alpha,q)=\min \{ \frac{1}{4},\frac{2-\alpha}{2\alpha},\frac
{q}{2}\}$ and $U_{0}$ is a constant depending on $C_{\phi},C_{\phi
,\mathcal{K}},M_{X}^{1},M_{\xi}^{1},M_{\zeta},M,I_{\Delta}$ given in (\ref{U}).
\end{lemma}

\begin{proof}
Step 1. Note that for any $(t,x)\in[\Delta,T]\times \mathbb{R}$ and
$(\tau,e)\in(-\varepsilon^{2},0)\times B(0,\varepsilon)$,
\[
S(\Delta,x,u_{\Delta}(t-\tau,x-e),u_{\Delta}(t-\Delta,\cdot-e))=0.
\]
Multiplying the above equality by $\zeta_{\varepsilon}(\tau,e)$ and
integrating with respect to $(\tau,e)$, from the concavity of the
approximation scheme (\ref{4.1}), we have for $(t,x)\in(\Delta,T]\times
\mathbb{R}$,
\begin{equation}%
\begin{split}
0  &  =\int_{-\varepsilon^{2}<\tau<0}\int_{|e|<\varepsilon}S(\Delta
,x,u_{\Delta}(t-\tau,x-e),u_{\Delta}(t-\Delta-\tau,\cdot-e))\zeta
_{\varepsilon}(\tau,e)ded\tau \\
&  =\int_{-\varepsilon^{2}<\tau<0}\int_{|e|<\varepsilon}\big(u_{\Delta}%
(t-\tau,x-e)-\mathbb{\tilde{E}}[u_{\Delta}(t-\Delta-\tau,x-e+\Delta^{1/\alpha
}\xi)]\big)\Delta^{-1}\zeta_{\varepsilon}(\tau,e)ded\tau \\
&  \leq \big(u_{\Delta}^{\varepsilon}(t,x)-\mathbb{\tilde{E}}[u_{\Delta
}^{\varepsilon}(t-\Delta,x+\Delta^{1/\alpha}\xi)]\big)\Delta^{-1}%
=S(\Delta,x,u_{\Delta}^{\varepsilon}(t,x),u_{\Delta}^{\varepsilon}(t,\cdot)).
\end{split}
\label{5.7}%
\end{equation}

Step 2. Since $u_{\Delta}^{\varepsilon}\in C_{b}^{\infty}([0,T]\times
\mathbb{R)}$, substituting $u_{\Delta}^{\varepsilon}$ into the consistency
property in Proposition \ref{prop} (iii), together with (\ref{5.6}) and
(\ref{5.7}), we can compute that%
\[
\partial_{t}u_{\Delta}^{\varepsilon}(t,x)-\sup \limits_{k_{\pm}\in K_{\pm}%
}\left \{  \int_{\mathbb{R}}\delta_{z}u_{\Delta}^{\varepsilon}(t,x)F_{k_{\pm}%
}(dz)\right \}  \geq-C_{\phi}M_{\zeta}(1+I_{\Delta})(1+\varepsilon^{-1}%
\Delta^{\frac{1}{2}})C(\varepsilon,\Delta),
\]
where $C(\varepsilon,\Delta)$ is defined in (\ref{5.4}). Then, the function%
\[
\bar{v}(t,x):=u_{\Delta}^{\varepsilon}(t,x)+C_{\phi}M_{\zeta}(1+I_{\Delta
})(1+\varepsilon^{-1}\Delta^{\frac{1}{2}})C(\varepsilon,\Delta)(t-\Delta)
\]
is a supersolution of (\ref{u_PIDE}) in $(\Delta,T]\times \mathbb{R}$ with
initial condition $\bar{v}(\Delta,x)=u_{\Delta}^{\varepsilon}(\Delta,x)$. In
addition,
\[
\underline{v}(t,x)=u(t,x)-C_{\phi}(M_{X}^{1}+M_{\xi}^{1})\Delta^{\frac
{1}{\alpha}}-C_{\phi}(1+I_{\Delta})(\varepsilon+\Delta^{\frac{1}{2}})
\]
is a viscosity solution of (\ref{u_PIDE}) in $(\Delta,T]\times \mathbb{R}$.
From (\ref{5.6}) and Lemma \ref{first interval estimate}, we can further
obtain%
\begin{align*}
\underline{v}(\Delta,x)  &  =u(\Delta,x)-C_{\phi}(M_{X}^{1}+M_{\xi}^{1}%
)\Delta^{\frac{1}{\alpha}}-C_{\phi}(1+I_{\Delta})(\varepsilon+\Delta^{\frac
{1}{2}})\\
&  =(u(\Delta,x)-u_{\Delta}(\Delta,x))+(u_{\Delta}(\Delta,x)-u_{\Delta
}^{\varepsilon}(\Delta,x))+u_{\Delta}^{\varepsilon}(\Delta,x)\\
&  \; \; \; \;-C_{\phi}(M_{X}^{1}+M_{\xi}^{1})\Delta^{\frac{1}{\alpha}%
}-C_{\phi}(1+I_{\Delta})(\varepsilon+\Delta^{\frac{1}{2}})\\
&  \leq u_{\Delta}^{\varepsilon}(\Delta,x)=\bar{v}(\Delta,x).
\end{align*}
By means of the comparison principle for PIDE (\ref{u_PIDE}) (see Proposition
5.5\ in \cite{NN2017}), we conclude that $\underline{v}(t,x)\leq \bar{v}(t,x)$
in $[\Delta,T]\times \mathbb{R}$, which implies for $(t,x)\in \lbrack
\Delta,T]\times \mathbb{R}$,
\begin{equation}
u-u_{\Delta}^{\varepsilon}\leq C_{\phi}[(M_{X}^{1}+M_{\xi}^{1})\Delta
^{\frac{1}{\alpha}}+(1+I_{\Delta})(\varepsilon+\Delta^{\frac{1}{2}}%
)+TM_{\zeta}(1+I_{\Delta})(1+\varepsilon^{-1}\Delta^{\frac{1}{2}%
})C(\varepsilon,\Delta)]. \label{5.8}%
\end{equation}

Step 3. Using\ (\ref{5.6}) and (\ref{5.8}), we have
\begin{align*}
u-u_{\Delta}  &  =(u-u_{\Delta}^{\varepsilon})+(u_{\Delta}^{\varepsilon
}-u_{\Delta})\\
&  \leq C_{\phi}[(M_{X}^{1}+M_{\xi}^{1})\Delta^{\frac{1}{\alpha}%
}+2(1+I_{\Delta})(\varepsilon+\Delta^{\frac{1}{2}})+TM_{\zeta}(1+I_{\Delta
})(1+\varepsilon^{-1}\Delta^{\frac{1}{2}})C(\varepsilon,\Delta)].
\end{align*}
Under Assumptions (A1)-(A3), we have $I_{\Delta}<\infty$, $R^{0}\leq4M$,
$R_{\Delta}^{1}\leq10C\Delta^{q}$, and $R_{\Delta}^{2}\leq16C\Delta^{q}$. In
the same way as Lemma \ref{lower bound}, by minimizing with respect to
$\varepsilon$, we can derive that for $(t,x)\in \lbrack \Delta,T]\times
\mathbb{R}$,
\[
u-u_{\Delta}\leq U_{0}\Delta^{\Gamma(\alpha,q)},
\]
where
\begin{equation}
U_{0}=C_{\phi}[M_{X}^{1}+M_{\xi}^{1}+4(1+I_{\Delta})+2TM_{\zeta}(1+I_{\Delta
})(2(1+M_{\xi}^{1})+4M+26C)] \label{U}%
\end{equation}
and $\Gamma(\alpha,q)=\min \{ \frac{1}{4},\frac{2-\alpha}{2\alpha},\frac{q}%
{2}\}$. Combining this and Lemma \ref{first interval estimate}, we obtain the
desired result.
\end{proof}

\end{document}